\documentclass[a4paper,11pt]{amsart}
\usepackage{amsmath,amsthm,amssymb,amsfonts,enumerate,color,enumerate}
\usepackage{hyperref}
\usepackage{verbatim}
\usepackage[english]{babel} 
\usepackage{color}
\usepackage{tikz}
\usepackage{graphicx}
\usepackage[thicklines]{cancel}

\newcommand{\red}[1]{\textcolor{black}{#1}}
\newtheorem{theorem}{Theorem}

\newtheorem{lemma}[theorem]{Lemma}

\newtheorem{proposition}[theorem]{Proposition}
\newtheorem{definition}[theorem]{Definition}

\newtheorem{remark}[theorem]{Remark}
\newtheorem*{theorem*}{Theorem}
\newcommand{\comm}[1]{}

\allowdisplaybreaks

\newcommand{\R}{\mathbb{{R}}}

\newcommand{\N}{\mathbb{{N}}}

\newcommand{\C}{\mathbb{{C}}}
\newcommand{\LL}{\mathcal{{L}}}

\newcommand{\MM}{\mathcal{{M}}}
\newcommand{\Sg}{\mathcal{{S}}}
\DeclareMathOperator{\ind}{ind}
\DeclareMathOperator{\mre}{Re} 
\DeclareMathOperator{\mim}{Im}
\DeclareMathOperator{\dist}{dist}
\DeclareMathOperator{\arc}{arc}
\DeclareMathOperator{\face}{face}
\DeclareMathOperator{\loc}{loc}
\DeclareMathOperator{\supp}{supp}
\DeclareMathOperator{\ess}{ess}
\DeclareMathOperator{\Tr}{Tr}

\title{The quasi-static plasmonic problem for polyhedra}
\author{Marta de Le\'on-Contreras}
\address{ Department of  Mathematical Sciences, Norwegian University of Science and Technology (NTNU),  \newline NO-7491 Trondheim, Norway
}
\email{marta.deleoncontreras@ntnu.no}
\author{Karl-Mikael Perfekt}
\address{ Department of  Mathematical Sciences, Norwegian University of Science and Technology (NTNU),  \newline NO-7491 Trondheim, Norway
}
\email{karl-mikael.perfekt@ntnu.no}
\begin{document}
		\begin{abstract}
		We characterize the essential spectrum of the plasmonic problem for polyhedra in $\R^3$. The description is particularly simple for convex polyhedra and permittivities $\epsilon < - 1$. The plasmonic problem is interpreted as a spectral problem through a boundary integral operator, the direct value of the double layer potential, also known as the Neumann--Poincar\'e operator. We therefore study the spectral structure of the double layer potential for polyhedral cones and polyhedra.
	\end{abstract}
	\maketitle
	
	\section{Introduction}
	\subsection{Background}
	Let $\Omega\subset \R^3$ be an open simply connected bounded polyhedron (\red{with flat faces and straight edges}), understood as an inclusion into infinite space with relative permittivity $\epsilon \in \C$, $\mre \, \epsilon < 0$. For a given function (or distribution) $g$ on $\partial \Omega$, the quasi-static plasmonic problem seeks a potential $U \colon \R^3 \to \C$, 
	$$U(x) = o(1), \qquad  |x| \to \infty,$$
	which is harmonic in $\Omega$ and its exterior,
	$$\Delta U (x) = 0, \qquad x \in \R^3 \setminus \partial \Omega,$$
	and satisfies
	$$\Tr_+ U = \Tr_- U, \qquad \epsilon \left( \frac{\partial}{\partial n} U\right)_+ - \left( \frac{\partial}{\partial n} U\right)_- = g$$
	on $\partial \Omega$. Here $\Tr_\pm U$ and $\left( \frac{\partial}{\partial n} U\right)_\pm$ denote the interior/exterior traces and limiting outward normal derivatives of $U$ on $\partial \Omega$. A value of $\epsilon$ for which there is a non-zero solution $U$ of the plasmonic problem with $g = 0$ is a plasmonic eigenvalue; the corresponding eigenfield $\nabla U$ is a static plasmon associated with the permittivity $\epsilon$. 
	
	If $\Omega$ is Lipschitz and $U$ is assumed to be of finite energy, then any plasmonic eigenvalue $\epsilon$ must satisfy that $\epsilon < 0$, since Green's formula implies that $ \int_{\R^3 \setminus \overline{\Omega}} |\nabla U|^2 \, dx = -\epsilon \int_{\Omega} |\nabla U|^2 \, dx$. Plasmonic problems, where $\mre \, \epsilon < 0$, appear as quasi-static approximations of electrodynamical problems where the scatterer is much smaller than the wavelength of the scattered electromagnetic wave, see \cite{ARYZ16} and \cite[Section~8]{HKR20}. If instead $\epsilon > 0$, or if $\mim \, \epsilon \neq 0$ and $\mre \, \varepsilon > 0$, then the described problem is an ordinary electrostatic or quasi-static problem, which is thoroughly studied and mostly very well understood.
	
	We will use layer potential operators to interpret and analyze the plasmonic problem of $\Omega$ as a spectral problem. Given a charge $u$ on $\partial \Omega$, the corresponding single layer potential of $-\Delta$ is given by 
	\begin{equation*}
	    \mathcal{S}u(x)=\frac{1}{4\pi}\int_{\partial\Omega}\mathcal{S}(x,y)u(y) \, dS(y)=\frac{1}{4\pi}\int_{\partial\Omega}\frac{u(y)}{|x-y|} \, dS(y), \qquad {x\in\R^3},  
	\end{equation*}
		where $dS$ denotes the standard surface measure on $\partial\Omega$.
 The Neumann--Poincar\'e operator on $\partial\Omega$, the direct value of the double layer potential of $u$,   is defined by
	\begin{align*}
	Ku(x)=-\frac{1}{4\pi}\int_{\partial\Omega} u(y)\frac{\partial}{\partial n_y}\frac{1}{|x-y|} \, dS(y)=\frac{1}{4\pi}\int_{\partial\Omega} u(y)\frac{(y-x)\cdot n_y}{|x-y|^3} \, dS(y), \qquad x\in\partial\Omega,
	\end{align*}
	 where $n_y$ is the outward normal vector at (almost every) $y \in \partial\Omega$. Inserting the ansatz
	 $$U = \mathcal{S}u$$
	 into the plasmonic problem yields the equation 
	 \begin{equation} \label{eq:plasmoniclaypot}
	 (K^\ast - \lambda) u = \frac{g}{1 - \epsilon}, \qquad \lambda = \frac{\epsilon + 1}{2(\epsilon - 1)},
	 \end{equation}
	 where the adjoint $K^\ast$ has been formed with respect to the $L^2(\partial \Omega)$-pairing.
	 	 Note that $\mre \, \epsilon < 0$ if and only if $|\lambda| < 1/2$.
	 For justification of the connection between the plasmonic problem and the spectral theory of $K$ in $L^p$, Sobolev, and Hardy space settings, see for example \cite{AFKRYZ18, EM04, HKL17, HP13}. For a treatment that includes classes of non-Lipschitz domains, see \cite{HMT10}.
	 
	 For smooth domains, the spectrum of the plasmonic problem consists solely of a sequence of eigenvalues, which for 3D domains is governed by a Weyl law \cite{AKMP20, MR19}. For domains with singularities, the plasmonic problem also exhibits essential spectrum (here interpreted via the connection with the Neumann--Poincar\'e operator). \red{To exemplify this, we recall that for a curvilinear polygon $\Omega$ in 2D, with interior angles $\beta_1, \ldots, \beta_J$, the  spectral picture of the analogous 2D Neumann--Poincar\'e operator is very well understood \cite{BHM20, BZ19, Mitreapolygon, Per21, PerfektPutinar}. As is typical for domains with singularities, the situation is highly dependent on the choice of function space. On the Sobolev space $H^{1/2}(\partial \Omega)$, the most physically meaningful choice, there is a self-adjoint realization of $K \colon H^{1/2}(\partial \Omega) \to H^{1/2}(\partial \Omega)$, and the essential spectrum is absolutely continuous and given by
     \begin{equation} \label{eq:NP2Dcurvispec}
     \sigma_{\ess}\left( K , H^{1/2}(\partial \Omega)\right)= \left\{x\in\R \, : \, |x|\le \max_{1\le j\le J   } \frac{|1-\beta_j/\pi|}{2} \right\}.
     \end{equation}
     On the other hand, the essential spectrum (in the sense of Fredholm operators) of $K \colon L^2(\partial \Omega) \to L^2(\partial \Omega)$ is a union of complex curves,
     $$\sigma_{\ess}\left( K , L^2(\partial \Omega)\right) = \bigcup_{1 \leq j \leq J} (\Sigma_{0, \beta_j} \cup \Sigma_{0, \beta_j}^-),$$
     where
  $$ \Sigma_{0, \beta_j} = \left \{\frac{1}{2}\frac{\sin((\pi-\beta_j)(\frac{1}{2}+i\xi))}{\sin(\pi(\frac{1}{2}+i\xi))} \, : \, -\infty \leq \xi \leq \infty \right\}$$
  and $\Sigma_{0, \beta_j}^- = - \Sigma_{0, \beta_j}$. Furthermore, outside the essential spectrum, the index of $K-\lambda$ is given by the winding number of $\lambda$ with respect to $\sigma_{\ess}\left( K , L^2(\partial \Omega)\right)$. The $L^2(\partial \Omega)$-theory can be understood through the lens of Mellin pseudodifferential operators, which we will briefly explain in Section~\ref{subsec:mellinpdo}.} Other types of singularities in 2D have also been considered \cite{BDTZ19}. Much less is known for 3D domains, \red{but analogous results for smooth conical singularities have been considered in \cite{BCR20} and \cite{HP18}, and for edges in \cite{Per19}.} The plasmonic problem has also been investigated numerically for some regular polyhedra \cite{HP13, Sih04}.
	 
	 For polyhedra $\Omega \subset \R^3$, the study of the (essential) spectral radius of $K \colon X \to X$ and the invertibility of $K + 1/2 \colon X \to X$ is a topic of very rich history; a vast variety of function spaces $X$ on $\partial \Omega$ have been considered. The invertibility of $K+1/2$ reflects the possibility of solving the Dirichlet problem in $\Omega$ with boundary data from $X$. We refer to \cite{Wend09} for an extensive survey, choosing here to only summarize the state of the art as it relates to the plasmonic problem.
	 
	 Rathsfeld \cite{Rat92} proved that $K + 1/2$ (appropriately modified at the edges) is invertible on the space $C(\partial \Omega)$ of continuous functions on $\partial \Omega$, for arbitrary polyhedra $\Omega$. To prove his result, Rathsfeld estimated the spectral radius on polyhedral cones using Mellin techniques -- note well the important correction that was made to this analysis in \cite{Rath95}. Elschner \cite{Elschner} refined the analysis further and proved that the essential spectral radius of $K$ is less than $1/2$ for a range of weighted Sobolev spaces on Lipschitz polyhedra $\partial \Omega$; \red{in Lemma~\ref{PropertiesH} we will recall some important details of Elschner's study.} Grachev and Maz'ya independently obtained the same result as Rathsfeld, and additionally established the invertibility of $K+1/2$ on weighted $L^p$-spaces for general polyhedra, see \cite{GM13, Maz12}. The Mellin techniques of \cite{Elschner, Rat92} were adapted to the study of other layer potential operators in \cite{Mitrea99}. However, it appears that the reasoning in the proof of \cite[Theorem~5]{Mitrea99} suffers from the same type of flaw as that of \cite[Lemma~1.5]{Rat92}, cf. Theorem~\ref{specK} and Remark~\ref{rmk:resolvent}. In \cite{Mitrea99} it was also proven that the essential spectral radius of $K \colon L^2(\partial \Omega) \to L^2(\partial \Omega)$ is less than $1/2$ if $\Omega$ has sufficiently small Lipschitz character; the {\it spectral radius conjecture} asks if this is true for all Lipschitz domains. The spectral radius conjecture is known to be true on the Sobolev space $H^{1/2}(\partial \Omega)$ \cite{Cost07, SW01}.
	
\subsection{Results} The main purpose of this article is to describe the essential spectrum  of $K \colon H^{1/2}(\partial \Omega) \to H^{1/2}(\partial \Omega)$, or, equivalently, of $K^\ast \colon H^{-1/2}(\partial \Omega) \to H^{-1/2}(\partial \Omega)$, for Lipschitz polyhedra $\Omega$. Note that in the layer potential formulation \eqref{eq:plasmoniclaypot} of the plasmonic problem, charges $u \in H^{-1/2}(\partial \Omega)$ correspond to potentials $U = \mathcal{S}u$ with finite energy, $\int_{\R^3} |\nabla U|^2 \, dx < \infty$. We will also investigate the spectra of $K \colon L^2_\alpha(\partial \Omega) \to L^2_\alpha(\partial \Omega)$ for certain weighted $L^2$-spaces, $L^2_\alpha(\partial \Omega)$. Our results in this latter setting will serve as crucial tools for our study in $H^{1/2}(\partial \Omega)$, but we also believe that they are of some independent interest. All of our results in the $L^2_\alpha(\partial \Omega)$-context are valid for arbitrary polyhedra, including non-Lipschitz polyhedra such as the interior of
$$([0,1] \times [0,1] \times [0,2]) \cup ([1,2] \times [0,2] \times [0,1]),$$
\red{which is a variant of the so-called ``two brick'' domain}.
To understand general bounded polyhedra, we first analyze the Neumann--Poincar\'e operator $K$ for polyhedral cones $\Gamma$ which locally coincide with $\partial \Omega$ around vertices. Assuming that $\Gamma$ has its vertex at the origin, we consider as in Figure~\ref{fig:polyhedralconep3} the spherical polygon
$$\gamma = \Gamma \cap S^2,$$
where $S^2$ denotes the two-dimensional unit sphere. In this case, $K$ is a Mellin operator with an operator-valued convolution kernel \cite{QN}, and this leads us to consider the (direct value of the) double layer potential operator $H(i\xi)$ on $\gamma$, $\xi \in \R$, formed with respect to the fundamental solution for $-\Delta_{S^2} + 1/4 + \xi^2$, where $\Delta_{S^2}$ is the Laplace--Beltrami operator of $S^2$ -- see Section~\ref{sec:mellinoncones}. 
\begin{figure}[ht]
	\centering
	\includegraphics[width=0.4\linewidth]{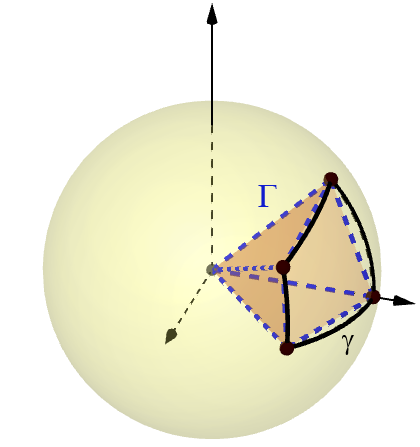}
	\caption{An illustration of a polyhedral cone $\Gamma$ and the corresponding spherical polygon $\gamma$.}
	\label{fig:polyhedralconep3}
\end{figure}

In Theorem~\ref{SpecHH1/2g} and Lemma~\ref{spectH}, we will describe the spectra of $H(i\xi) \colon H^{1/2}(\gamma) \to H^{1/2}(\gamma)$ and $H(i\xi) \colon L^2_\alpha(\gamma) \to L^2_\alpha(\gamma)$, $0 \leq \alpha < 1$, where
$$L^2_\alpha(\gamma) = L^2(\gamma, q^{-\alpha} \, d\omega),$$ $d\omega$ denoting the arc length measure on $\gamma$ and $q(\omega)$ a quantity comparable to the distance from $\omega \in \gamma$ to the corners of $\gamma$. In particular, the essential spectrum in the former case is given by
$$\sigma_{\ess}\left(H(i\xi), H^{1/2}(\gamma)\right)= \left\{x\in\R \, : \, |x|\le \max_{1\le j\le J   } \frac{|1-\beta_j/\pi|}{2} \right\},$$
where $\beta_1, \ldots, \beta_J$ denotes the internal angles of $\gamma$, \red{cf. \eqref{eq:NP2Dcurvispec}}. The remaining part of the spectrum $\sigma\left(H(i\xi), H^{1/2}(\gamma)\right)$ consists of real eigenvalues. Let 
\begin{equation*} 
\Lambda^\ast = \{\lambda \, : \, \lambda \textnormal{ is an isolated eigenvalue of } H(z) \colon H^{1/2}(\gamma) \to H^{1/2}(\gamma), \textnormal{ for some} \mre z = 0\}.
\end{equation*}
Similarly, for $0 \leq \alpha < 1$ we introduce
\begin{equation*} 
\Lambda^\alpha = \{\lambda \, : \, \lambda \textnormal{ is an isolated eigenvalue of } H(z) \colon L^2_\alpha(\gamma) \to L^2_\alpha(\gamma), \textnormal{ for some} \mre z = 0\}.
\end{equation*}
It turns out that both of these sets are real, and that they can be equivalently formed by considering isolated eigenvalues of the $L^2(\gamma)$-adjoint operators $H^\ast(z) \colon H^{-1/2}(\gamma) \to H^{-1/2}(\gamma)$ and $H^\ast(z) \colon L^2_{-\alpha}(\gamma) \to L^2_{-\alpha}(\gamma)$, respectively. It is an important observation that eigenfunctions to eigenvalues $\lambda$ of $H^\ast(i\xi)$ correspond to potentials $U$ on $S^2$ such that
$$(-\Delta_{S^2} + 1/4 + \xi^2) U (\omega) = 0, \qquad \omega \in S^2 \setminus \gamma,$$
and
$$\Tr_+ U = \Tr_- U, \qquad \epsilon \left( \frac{\partial}{\partial n} U\right)_+ = \left( \frac{\partial}{\partial n} U\right)_-$$
on $\gamma$, where $\lambda$ and $\epsilon$ are related as in \eqref{eq:plasmoniclaypot}, see Lemma~\ref{eigenvH*}.

Before we can discuss our main results, we need to introduce some additional notation. For Lipschitz polyhedral cones $\Gamma$, we introduce,  following \cite[Section~4]{Per19}, $\mathcal{E}(\Gamma)$ as a space of distributions on $\Gamma$ with norm given by 
$$\|f\|^2_{\mathcal{E}(\Gamma)} = \langle \mathcal{S}f, f \rangle_{L^2(\Gamma)}.$$
By results of \cite{Per19}, $\mathcal{E}(\Gamma)$ is isomorphic to the $L^2(\Gamma)$-dual of the homogeneous Sobolev space $\dot{H}^{1/2}(\Gamma)$. We will see that $K^\ast \colon \mathcal{E}(\Gamma) \to \mathcal{E}(\Gamma)$ is self-adjoint, and that $\mathcal{E}(\Gamma)$ is the correct space to consider for localization to $H^{-1/2}(\partial \Omega)$. We let $\hat{\Gamma}$ denote the interior of $\Gamma$ (which coincides locally with $\Omega$ around vertices), and we understand that $[c, d) = (c, d] = \emptyset$ if $c > d$. Let $j = j^\ast$ be the index which maximizes $|1 - \beta_j/\pi|$.
\begin{theorem*}
	Let $K$ be the Neumann--Poincar\'e operator of a Lipschitz polyhedral cone $\Gamma$. Then 
	\begin{equation*}\sigma(K^*, \mathcal{E}(\Gamma)) = \sigma_{\ess}(K^*, \mathcal{E}(\Gamma)) = \left[-\frac{|1 - \beta_{j^\ast}/\pi|}{2},\frac{|1 - \beta_{j^\ast}/\pi|}{2} \right]\cup \Lambda^\ast.
	\end{equation*}	
	Furthermore, there are two numbers $0 \leq \mu_\pm < 1/2$ such that
	\begin{equation*} 
	\Lambda^\ast = \left[-\mu_-, -\frac{|1 - \beta_{j^\ast}/\pi|}{2}\right) \cup \left(\frac{|1 - \beta_{j^\ast}/\pi|}{2}, \mu_+\right].
	\end{equation*}
	If $\hat{\Gamma}$ is convex, then $\mu_- \leq \mu_+$ and
	$$\mu_+ = \max \sigma(H(0), H^{1/2}(\gamma)) = \max \sigma(H^\ast(0), H^{-1/2}(\gamma)).$$
\end{theorem*}
\red{Loosely speaking, it is the edges of $\Gamma$ which give rise to the 
subinterval 
$$\left[-\frac{|1 - \beta_{j^\ast}/\pi|}{2},\frac{|1 - \beta_{j^\ast}/\pi|}{2} \right] \subset \sigma_{\ess}(K^*, \mathcal{E}(\Gamma)),$$
cf. \cite{Per19}. When $\hat{\Gamma}$ is not convex, it can happen that $\mu_- > |1 - \beta_{j^\ast}/\pi|/2$, as can be seen by employing the idea of the proof of Theorem~\ref{thm:coneconvex}. In the convex case, the role of $\mu_-$ is rather mysterious, and we do not know whether it might actually be the case that $\mu_- = 0$. }

We will prove a similar result for $L^2_\alpha(\Gamma) = L^2(\R_+, dr) \otimes L^2_\alpha(\gamma)$, $0 \leq \alpha < 1$. In fact, our study of $K \colon L^2_\alpha(\Gamma) \to L^2_\alpha(\Gamma)$ will inform our study of $K^\ast \colon \mathcal{E}(\Gamma) \to \mathcal{E}(\Gamma)$. In particular, we will show that
$$\Lambda^\ast = \bigcup_{0 \leq \alpha < 1} \Lambda^\alpha,$$
and an associated regularity result: if $g \in H^{-1/2}(\gamma)$ is an eigenvector to an eigenvalue $\lambda$ of $H^\ast(i\xi)$ with $|\lambda| > |1 - \beta_{j^\ast}/\pi|/2$, $H^\ast(i\xi) g = \lambda g$, then $g \in L^2_{-\alpha}(\gamma)$ for sufficiently large $\alpha < 1$. However, in this weighted $L^2$-setting, the set $\sigma(K, L^2_\alpha(\Gamma)) \setminus \Lambda^\alpha$ also contains complex points, which we are only able to partially describe. We therefore defer precise statements to Theorems~\ref{specK} and \ref{contessspecK}. 

One approach to the localization to $\partial \Omega$ is via the machinery of $b$-calculus \cite{Grie01, Mel81, Schulze91}, which relies on the construction of appropriate algebras of pseudo-differential operators. \red{It may be possible to treat curvilinear polyhedra using such techniques.} However, we shall take an alternative, rather direct approach to localization, \red{staying within our scope of polyhedral domains (with flat faces).} In one direction, we will construct Weyl sequences on $\partial \Omega$ in a procedure that seems applicable to a wider range of problems. We will prove complete localization results for both $L^2_\alpha(\partial \Omega)$, $0 \leq \alpha < 1$, and $H^{1/2}(\partial \Omega)$, where $L^2_\alpha(\partial \Omega)$ is defined in analogy with $L^2_\alpha(\Gamma)$. We have chosen to state the result only in the case of $H^{1/2}(\partial \Omega)$ here, deferring the remaining statement to Theorem~\ref{locarg1}.

\begin{theorem*}
	Let $K$ be the Neumann--Poincar\'e operator of a Lipschitz polyhedron $\partial \Omega$. For each vertex of $\partial \Omega$, let $K_i$ denote the Neumann--Poincar\'e operator of the corresponding tangent polyhedral cone $\Gamma_i$, $i = 1, \ldots, I$. Then, for $\lambda \in \mathbb{C}$, $K - \lambda$ is Fredholm on $H^{1/2}(\partial \Omega)$ if and only if $K_i^\ast - \lambda$ is invertible on $\mathcal{E}(\Gamma_i)$ for every $i=1,\dots,I$. That is,
	$$\sigma_{\ess}(K, H^{1/2}(\partial \Omega)) = \bigcup_{1 \leq i \leq I} \sigma(K_i^\ast, \mathcal{E}(\Gamma_i)).$$
\end{theorem*}
As in the definition of $\mathcal{E}(\Gamma)$, it is via the single layer potential possible to endow $H^{-1/2}(\partial \Omega)$ with a scalar product that makes $K^\ast \colon H^{-1/2}(\partial \Omega) \to H^{-1/2}(\partial \Omega)$ into a self-adjoint operator, see Section~\ref{subsec:energyspace}. Therefore the remaining non-essential spectrum of $K \colon H^{1/2}(\partial \Omega) \to H^{1/2}(\partial \Omega)$ consists of a sequence of isolated eigenvalues. \red{Typically, eigenvalues can appear in the localization of operators, see \cite{LPS22} for a relevant illustration. However, we are not aware of any specific examples relevant to our setting.}

Our description of $\sigma_{\ess}(K, H^{1/2}(\partial \Omega)) \cap \R_+$ is particularly simple for convex polyhedra, since one only needs to consider the double layer potential of $-\Delta_{S^2} + 1/4$ to compute this interval. Note that spectral parameters $\lambda \in (0, 1/2)$ correspond to permittivities satisfying $\epsilon < -1$; in \cite[Section~8]{HKR20}, it is suggested that $\epsilon < -1$ is likely to be a necessary condition for the existence of surface plasmon waves on $\partial \Omega$. Finally, we remark that the plasmonic problem for a cube is of importance to nanoplasmonics \cite{Fuchs, grillet2011plasmon, HP13, Langbein, ruppin1996plasmon}. The numerics of \cite{HP13} suggest that
$$\mu_- = 0, \qquad \mu_+ \approx 0.34726$$
when $\hat{\Gamma}$ is an octant of $\R^3$.

\subsection{Organization} In Section~\ref{sec:prelim} we define the function spaces already mentioned, we recall some elements of Fredholm theory and extrapolation of compact operators, and we discuss Mellin operators with operator-valued convolution kernels. In Section~\ref{sec:mellinoncones} we examine the relationship between the Neumann--Poincar\'e operator on a polyhedral cone $\Gamma$ and the double layer potential operators $H(z)$ on the associated spherical polygon $\gamma$. Section~\ref{sec:L2a} contains all of our theory concerning $L^2_\alpha(\partial \Omega)$, while Section~\ref{sec:energy} treats the energy space case.


\section{Notation and Preliminaries} \label{sec:prelim}

\subsection{Polyhedral cones $\Gamma$} \label{subsec:polyhedral} \red{Throughout the paper, $\Omega$ will denote a simply connected bounded polyhedron with straight faces}.
	The localization of $K$ to a corner of $\partial \Omega$ leads us to consider integral operators on the boundary $\Gamma$ of an infinite polyhedral cone $\hat{\Gamma}$ which locally coincides with $\Omega$. Without loss of generality, we may assume that $\Gamma$ has its vertex at the origin. The faces of $\Gamma$ are open plane sectors $F_j$, $j=1,\dots,J$; the edges of $\Gamma$ are denoted by $\upsilon_j$. 
	
	We shall denote by $\gamma$ the intersection of $\Gamma$ with the two-dimensional unit sphere $S^2$. That is, $\gamma$ is a spherical polygon consisting of the circular arcs $\gamma_j=S^2\cap F_j$ and the corner points $E_j=\upsilon_j\cap S^2$, $j=1,\dots, J$. Let $\hat{\gamma} = S^2 \cap \hat{\Gamma}$. Then $\partial\hat{\gamma}=\gamma$ and $\hat{\Gamma} =\R_+\hat{\gamma}$ is the \red{polyhedral} cone with base $\hat{\gamma}$.

Let $C^\infty_{\arc}(\gamma)$ be the set of all Lipschitz continuous functions on $\gamma$ whose restrictions to $\gamma_j$ belong to $C^\infty(\bar{\gamma}_j
)$, $j=1,\dots,J$. In the vicinity of each corner $E_j$, we parametrize the adjacent arcs $\gamma_{j-1}$ and $\gamma_j$ by the arc lengths $s = s_{j-1}$ and $s = s_j$ from $E_j$. We fix a function $q \in C^\infty_{\arc}(\gamma)$ such that $q(\omega) = s$ for $\omega = \omega(s)$ in a neighborhood of each corner $E_j$, non-vanishing except at corner points. For $\alpha \in \R$, we let 
\[L^2_{\alpha}(\gamma)= L^2(\gamma, q^{-\alpha} \, d\omega) = \left \{f \, : \, \int_{\gamma}|f|^2q^{-\alpha} \, d \omega <\infty \right\}, \]
where $d\omega$ denotes the arc length measure on $\gamma$, and 
$$L^2_{\alpha}(\Gamma)= L^2(\R_+, dr)\otimes L^2_{\alpha}(\gamma) = L^2(\R_+ \times \gamma, q^{-\alpha} \, dr \, d\omega).$$
Note that the usual $L^2$-space $L^2(\Gamma)$ coincides with $L^2(\R_+, r \, dr)\otimes L^2_{0}(\gamma)$.
\subsection{Weighted $L^2$-spaces on $\partial \Omega$} \label{subsec:weightedL2}
Throughout, $\partial\Omega$ will denote the boundary of \red{the} simply connected bounded polyhedron $\Omega\subset \R^3$ with vertices $\widetilde{E}_i$, $i=1,\dots, I$ and faces $\widetilde{F}_j, $ $j=1,\dots,J.$ For each  $i=1,\dots,I$, let $\Gamma_i$ be the tangent polyhedral cone to $\partial\Omega$ at the corner $\widetilde{E}_i.$
 We define $C^\infty_{\face}(\partial\Omega)$ as the space of Lipschitz continuous functions on $\partial\Omega$ that are $C^\infty $ on the closure of each face $\widetilde{F}_j$. By a compactness argument, we can choose a partition of unity $\{\varphi_i\}_{i=1}^I \subset C^\infty_{\face}(\partial\Omega)$ on $\partial\Omega$, such that $\varphi_i \equiv 1$ in a neighborhood of $\widetilde{E}_i$, $\varphi_i \equiv 0$ in a neighborhood of $\bigcup_{j \neq i} \widetilde{E}_j$, and $\supp \varphi_i \subset \Gamma_i$.  Then, given a function $f$ on $\partial \Omega$, we can naturally understand $\varphi_i f$ as a function on $\Gamma_i$. We define, for $\alpha<1$, the space $L^2_\alpha(\partial\Omega)$ as the completion of $C^\infty_{\face}(\partial\Omega)$ in the norm
$$
\|f\|_{L^2_\alpha(\partial\Omega)}^2 =\sum_{1\le i\le I}\|\varphi_i f\|_{L^2_\alpha(\Gamma_i)}^2.
$$

\subsection{The energy spaces on $\partial \Omega$ and $\Gamma$} \label{subsec:energyspace}
Following an idea that dates back to Poincar\'e \cite{Cost07, KPS07}, the energy space  $\mathcal{E}(\partial\Omega)$ is introduced as the Hilbert space obtained by completing $L^2(\partial\Omega)$ in the positive definite scalar product
\begin{align*}
\langle f, g\rangle_{\mathcal{E}(\partial\Omega)}=\langle \Sg f,g\rangle_{L^2(\partial\Omega)},
\end{align*}
where $\Sg$ denotes the single layer potential on $\partial\Omega.$ The reason for introducing the energy space is that $K^\ast \colon \mathcal{E}(\partial \Omega) \to \mathcal{E}(\partial \Omega)$ is self-adjoint, owing to the Plemelj formula $\Sg K^\ast = K \Sg$.

When discussing the energy space, we will for technical reasons assume that $\Omega$ is Lipschitz. That is, $\partial \Omega$ is locally the graph of a Lipschitz function whose epigraph locally coincides with $\overline{\Omega}$. Under this assumption, $\mathcal{E}(\partial\Omega)$ is a space of distributions on $\partial \Omega$ which is isomorphic to the Sobolev space $H^{-1/2}(\partial\Omega)$ of index $-1/2$ along $\partial \Omega$ \cite{SW01}, with equivalent norms,
$$
\|f\|_{H^{-1/2}(\partial\Omega)}^2\simeq \|f\|^2_{\mathcal{E}(\partial \Omega)} = \langle \Sg f,f\rangle_{L^2(\partial\Omega)}.
$$

When $\Gamma$ is a Lipschitz polyhedral cone, we similarly introduce the energy space $\mathcal{E}(\Gamma)$ as the completion of the space of compactly supported $L^2(\Gamma)$-functions in the scalar product
\begin{equation*}
\langle f, g\rangle_{\mathcal{E}(\Gamma)}=\langle \Sg f,g\rangle_{L^2(\Gamma)},
\end{equation*}
where $\Sg$ now denotes the single layer potential on $\Gamma$. In this case, $\mathcal{E}(\Gamma)$ coincides with the dual of the fractional homogeneous Sobolev space $\dot{H}^{1/2}(\Gamma)$ on $\Gamma$ \cite[Theorem~14]{Per19}.

Now let $\Gamma_i$, $i=1,\dots,I,$ be the tangent polyhedral cones to $\partial\Omega$ at the corners of $\partial \Omega$, and let $\{\varphi_i\}_{i=1}^I$ be the partition of unity on $\partial \Omega$ described in Section~\ref{subsec:weightedL2}. Then, for $f \in L^2(\partial \Omega)$,
$\supp \varphi_if\subset \Gamma_i\cap\partial\Omega$, and therefore
\begin{equation*}
\|\varphi_i f\|_{\mathcal{E}(\Gamma_i)}^2=\langle \Sg(\varphi_i f),\varphi_if\rangle_{L^2(\Gamma_i)}=\langle \Sg (\varphi_if),\varphi_if\rangle_{L^2(\partial\Omega)}=\|\varphi_i f\|_{\mathcal{E}(\partial\Omega)}^2\simeq\|\varphi_i f\|_{H^{-1/2}(\partial\Omega)}^2.
\end{equation*}
By density and the fact that each $\varphi_i$ is a multiplier of $H^{-1/2}(\partial\Omega) \simeq \mathcal{E}(\partial \Omega)$, it follows that
$$
\|f\|_{\mathcal{E}(\partial\Omega)}^2 \simeq \sum_{1\le i\le I}\|\varphi_i f\|_{\mathcal{E}(\Gamma_i)}^2, \qquad f \in \mathcal{E}(\partial\Omega).
$$
\subsection{Fredholm theory} \label{subsec:spectraltheory}
\red{Let $X$ and $Y$ be Banach spaces and let $T$ be a bounded linear operator from $X$ to $Y$, that is, $T\in \mathcal{L}(X,Y)$. Let $\alpha(T)=\dim \text{Ker} (T)$ and $\beta(T)=\dim(Y/\text{Ran}(T)),$ where $\text{Ker} (T)\subseteq X$  and $\text{Ran}(T)\subseteq Y$ denote the nullspace and the range of $T$, respectively. We say that $T$ is a {\it Fredholm operator} if $\text{Ran}(T)$ is closed, $\alpha(T)<\infty$ and $\beta(T)<\infty$. On the other hand,  $T$ is a {\it upper semi-Fredholm operator} if $\text{Ran}(T)$ is closed and $\alpha(T)<\infty$, whereas  $T$ is a {\it lower semi-Fredholm operator} if $\text{Ran}(T)$ is closed and  $\beta(T)<\infty$. If the operator $T$ is either upper or lower semi-Fredholm, we shall say that it is {\it semi-Fredholm}. We will now recall some elements of Fredholm theory that will be useful for us. For a complete treatment, see \cite{Schechter}.}

\red{The following criterion is very useful.
\begin{proposition}
    Let $T\in \mathcal{L}(X,Y)$. Then $T$ is upper semi-Fredholm if and only if there is a Banach space $Z$, a compact operator $S:X\to Z$, and a constant $C>0$ such that
    $$
    \|x\|_X\le C \|Tx\|_Y+\|Sx\|_Z, \qquad \forall x\in X.
    $$
\end{proposition}
A fundamental quantity associated with a (semi-)Fredholm operator is its {\it index}
$$\ind(T)=\alpha(T)-\beta(T).$$ 
 \begin{proposition}
     Let $X,Y$ be Banach spaces and $T\in \mathcal{L}(X,Y)$ be a semi-Fredholm operator. If $K$ is  a compact operator from $X$ to $Y$, then $T+K$ is also semi-Fredholm and $\ind(T+K)=\ind(T).$
 \end{proposition}}
 \red{Furthermore, the composition of Fredholm operators $T$ and $S$ is again Fredholm and $\ind(TS)=\ind(T)+\ind(S)$. This formula is also true for semi-Fredholm operators, as long as the right-hand side makes sense.}
 
 \red{
For an operator $T\in \mathcal{L}(X) = \mathcal{L}(X,X)$ we shall call
 $$
 \sigma_{\ess}(T,X):=\{\lambda\in\C: \, \lambda I-T \text{ is not Fredholm}\}
 $$
 the {\it essential spectrum} of $T$. A sequence $\{x_n\}_{n\in\N}$ such that $\|x_n\|_X=1$ for all $n$, $x_n\to 0$ weakly in $X$, and $\|(T-\lambda)x_n\|_X \to 0$, as $n\to\infty,$ is called a {\it Weyl sequence}. If $\{x_n\}_{n\in\N}$ is a Weyl sequence for an operator $T\in \mathcal{L}(X)$ and $\lambda\in\C$, then $\lambda\in \sigma_{\ess}(T,X).$ The converse is also true when $X$ is a Hilbert space and $T$ is self-adjoint.}

\subsection{Extrapolation of compactness} \label{subsec:extrapolation}
\red{When treating the energy space case in Section \ref{sec:energy} we will sometimes rely on the extrapolation result of Cwikel, \cite{Cwikel}, in order to establish the compactness of certain operators.}

\red{For a compatible couple of Hilbert spaces $(A_0,A_1)$ and $0<\theta<1$, let $(A_0,A_1)_{\theta}$ denote the real interpolation space between $A_0$ and $A_1$. Since we are dealing with Hilbert spaces, we will always assume that $q=2$ in the real interpolation method, omitting it from the notation. Also note that complex and real interpolation coincide in the Hilbert space case, see \cite{CHM}.
\begin{definition}[K-method]
Let $(A_0, A_1)$ be a compatible couple of Hilbert spaces and $0<\theta<1.$ The space $(A_0,A_1)_{\theta}$ consists of all $f\in A_0+A_1$ for which the functional
$$
\|f\|_\theta:=\left(\int_0^\infty (t^{-\theta}K(f,t;A_0,A_1))^2\frac{dt}{t} \right)^{1/2}
$$
is finite, 
where $$K(f,t;A_0,A_1):=\inf\{\|f_0\|_{A_0}+t\|f_1\|_{A_1}:\: f=f_0+f_1,\: f_0\in A_0, \;f_1\in A_1\}.$$
\end{definition}}
\red{We will use extrapolation in the scale of Sobolev spaces, see \cite{MMTrans}.}
\red{\begin{proposition}
Let $\Sigma$ be a bounded Lipschitz domain, $0<\theta<1$ and  $0\le s_0,s_1\le 1$, with $s_0\neq s_1.$ Then
\begin{align*}
    (H^{s_0}(\partial\Sigma), H^{s_1}(\partial\Sigma))_\theta&=H^{s}(\partial\Sigma)\\
     (H^{-s_0}(\partial\Sigma), H^{-s_1}(\partial\Sigma))_\theta&=H^{-s}(\partial\Sigma),
\end{align*}
where $s=(1-\theta)s_0+\theta s_1$.
\end{proposition}
Finally, we state the extrapolation result.
\begin{theorem}\cite[Theorem 1.1]{Cwikel}.
Let $(A_0,A_1)$ and $(B_0,B_1)$ be two compatible couples of Banach spaces and let $T:(A_0,A_1)\to (B_0,B_1)$ be a linear operator such that $T:{A}_0\to {B}_0$ is bounded and $T:A_1\to B_1$ is compact. Then $T:(A_0,A_1)_{\theta} \to(B_0,B_1)_{\theta}$ is compact for every $\theta\in (0,1)$.
\end{theorem}}

\subsection{Mellin convolution operators} \label{subsec:mellinconv}
	Recall that the Mellin transform of a sufficiently nice function $f \colon \mathbb{R}_+ \to \mathbb{C}$ is defined by
$$
\mathcal{M}f(z)=\int_0^\infty f(t)t^{z} \, \frac{dt}{t}, \qquad z \in \mathbb{C}.
$$
Up to a scaling factor,  the Mellin transform induces a unitary map 
$$\MM: L^2(\R_+,t^{2m-1}dt)\rightarrow L^2(\{\mre z=m\}, |dz|), \qquad m \in \R.$$
The Mellin convolution of appropriate functions $f$ and $g$ is given by
$$
(f\ast g)(s)=\int_0^\infty f(s/t)g(t)\frac{dt}{t},
$$
and
$$
\MM ( f \ast g)(z)=\MM f(z) \MM g(z).
$$
Referring back to the notation for polyhedral cones introduced in Section~\ref{subsec:polyhedral}, for a function $f$ on $\Gamma$ or $\hat{\Gamma}$, we shall also write $\MM f(z)$ for the Mellin transform in the radial variable,
$$(\MM f(z))(\omega) = \int_0^\infty f(r\omega) r^{z} \, \frac{dr}{r}.$$
Here $x = r\omega$ has been written in spherical coordinates; $r = \dist(0, x)$ and $\omega \in \gamma$ or $\omega \in \hat{\gamma}$.

We say that an operator $T \colon C_c^\infty(\R_+ \times \cup \gamma_j) \to L^2_{\loc}(\Gamma)$ has an operator-valued convolution kernel $T(t, \omega, \omega')$ if $T$ is of the form
$$T u (r\omega) = \int_{\mathbb{R}_+ \times \gamma} T(r/r', \omega, \omega') u(r'\omega')\, d\omega' \frac{dr'}{r'}.$$
 When convergent, we shall then denote by $\MM T(z) \colon C^\infty_c(\cup \gamma_j) \to L^2(\gamma)$ the operator given by the integral kernel
$$(\MM T(z))(\omega, \omega') = \int_0^\infty T(t, \omega, \omega') t^{z} \, \frac{dt}{t}.$$
We also make use of the analogous terminology and notation for operators $T \colon C^\infty_c(\hat{\Gamma}) \to L^2_{\loc}(\hat{\Gamma})$.

Consider the multiplication operator $M_{r^{1/2}}$ defined by
$$M_{r^{1/2}} f(r\omega)=r^{1/2} f(r\omega).$$
 Observe that if $T$ is a Mellin convolution operator with kernel $T(t,\omega,\omega')$, then $M_{r^{1/2}}TM_{r^{-1/2}}$ is also a Mellin convolution operator with kernel $t^{1/2}T(t,\omega,\omega')$. Therefore, at least formally,
\begin{equation*}
\MM(M_{r^{1/2}}TM_{r^{-{1/2}}}f)(z)= \MM(M_{r^{1/2}}TM_{r^{-{1/2}}})(z)\MM f(z)=\MM T(z+1/2)\MM f(z).
\end{equation*}
Furthermore,
\begin{equation} \label{eq:mellinshifted}
\MM M_{r^{1/2}}: L^2_\alpha(\Gamma) \to L^2(\{\mre z =0\}, |dz|) \otimes L^2_\alpha(\gamma)
\end{equation}
is unitary up to a scaling factor, where, for sufficiently nice $f \in L^2_\alpha(\Gamma)$ and $\mre z = 0$,
$$
(\MM M_{r^{1/2}}f(z))(\omega)=\int_0^\infty t^{1/2}f(t\omega)t^{z} \, \frac{dt}{t}=(\MM f(z+1/2))(\omega).
$$

Let $\alpha < 1$. Via \eqref{eq:mellinshifted}, any Mellin convolution operator that extends to a bounded operator $T \colon L^2_\alpha(\Gamma) \to L^2_\alpha(\Gamma)$ is therefore unitarily equivalent to 
$$I \otimes \MM T(i\xi+1/2) \colon L^2(\mathbb{R}, d\xi) \otimes L^2_\alpha(\gamma) \to L^2(\mathbb{R}, d\xi) \otimes L^2_\alpha(\gamma).$$
With this in mind, we record the following elementary lemma for future use. We provide a proof to preserve the concrete presentation pursued in this section.
	\begin{lemma} \label{lem:tensorbdd}
	Let $\mathcal{H}$ be a separable Hilbert space and let $\{A(\xi)\}_{\xi \in \mathbb{R}}$ be a strongly measurable family of operators $A(\xi) \colon \mathcal{H} \to \mathcal{H}$ such that $\sup_{\xi}\|A(\xi)\| < \infty.$ Then 
	$$I\otimes A(\xi) \colon L^2(\R,d\xi)\otimes\mathcal{H} \to L^2(\R,d\xi)\otimes\mathcal{H}$$
	defines a bounded operator, and
	$$\| I\otimes A(\xi)\| \leq \sup_{\xi\in\R}\|A(\xi)\|.$$
	\end{lemma}
	
	\begin{proof}
	Fix orthonormal bases $\{e_j\}_j$ and $\{f_k\}_k$ of $L^2(\R,d\xi)$ and $\mathcal{H}$, respectively. Then every $h\in L^2(\R,d\xi)\otimes\mathcal{H}$ can be written $h =\sum_{j,k}a_{j,k}e_j \otimes f_k$, where $\|h\|^2 = \sum_{j,k} |a_{j,k}|^2$. First assume that the sum is finite. Then
	\begin{align*}
	    	\| (I\otimes A(\xi))h\|_{L^2(\R,d\xi)\otimes\mathcal{H}}^2&=\int_{\R}\Big\|A(\xi)\Big(\sum_{j,k}a_{j,k}e_j(\xi) f_k \Big)\Big\|_{\mathcal{H}}^2 \, d\xi\\
	    	&\le\sup_{\xi\in\R}\|A(\xi)\|^2\int_{\R}\Big\|\sum_{j,k}a_{j,k}e_j(\xi) f_k \Big\|_{\mathcal{H}}^2 \, d\xi\\
	    	&=\sup_{\xi\in\R}\|A(\xi)\|^2\int_{\R}\sum_k\Big|\sum_{j}a_{j,k}e_j(\xi) \Big|^2 \, d\xi\\
	    	&=\sup_{\xi\in\R}\|A(\xi)\|^2\sum_{j,k}|a_{j,k}|^2=\sup_{\xi\in\R}\|A(\xi)\|^2\|h\|^2.
	\end{align*}
	The statement now follows in the usual manner, extending by density the domain of definition of $I \otimes A(\xi)$ from finite sums to arbitrary $h$.
	\end{proof}
\subsection{Localizations of Mellin operators} \label{subsec:mellinpdo}
\red{There is a well-developed theory also of pseudo-differential operators of Mellin type \cite{Elschner1987, Lewis1990, Mitreapolygon}. For our purposes, we only need to apply the results for a localized (scalar) Mellin convolution operator. The formulation described here can be deduced from Theorem~1 in \cite{Lewis1990} and the subsequent remark.}

\red{Suppose that $T$ is a Mellin convolution operator 
$$Tf(s) = \int_0^\infty f(t) a(s/t) \, \frac{dt}{t}, \qquad s > 0,$$
with a convolution kernel $a(s/t)$ for which there are $\alpha < 0 < \beta$ such that 
\begin{equation*}
\sup_{\alpha \leq \mre z \leq \beta}  \left|(1+|z|)^{n+1} \left(\frac{d}{dz}\right)^n \mathcal{M} a(z) \right| < \infty, \qquad n=0,1,2,\ldots
\end{equation*}
Let $\varphi \in C^\infty([0,1])$ be a cut-off function such that $\varphi \equiv 1$ in a neighborhood of $0$ and $\varphi \equiv 0$ in a neighborhood of $1$. Then the essential spectrum of $\varphi T \varphi \colon L^2([0,1], dt/t) \to L^2([0,1], dt/t)$ is given by
\begin{equation*}
\sigma_{\ess}(\varphi T \varphi, L^2([0,1], dt/t)) = \{ \mathcal{M}a(i\xi) \, : \, -\infty \leq \xi \leq \infty \}.
\end{equation*}
Furthermore, when $\lambda$ does not belong to this curve, then the Fredholm index of $\varphi T \varphi - \lambda$ coincides with the winding number of $\lambda$ with respect to the essential spectrum.}
\subsection{Asymptotics of certain Mellin transforms}
In this subsection we record the asymptotics of the Mellin transforms of some functions that appear in connection with our analysis of single and double layer potentials on polyhedral cones. Exact formulas in terms of known special functions can be deduced from \cite[p. 257, formula (9.7.5)]{Lebedev} and \cite[p. 25, formula 2.64]{Oberhettinger}.

For $-1 \leq a < 1$ and $-3/2 < \mre z < 3/2$, we have that
\begin{equation}\label{eq:KMellinasymp}
 \int_0^\infty \frac{t^{z+3/2}}{(t^2-2at+1)^{3/2}}\frac{dt}{t} = \frac{1}{1-a} + O(1+|\log(1-a)|),
 \end{equation}
where the implied constant depends on $z$. 

For $-1 \leq a < 1$ and $1 < \mre z < 2$, we have that
\begin{equation}\label{eq:PhiMellinasymp}
\int_0^{\infty} \frac{t^{z-1}}{(t^2-2at+1)^{1/2}}\frac{dt}{t} = - \log\left( \frac{1-a}{2}\right) + O(1),
\end{equation}
where, again, the implied constant depends on $z$.
		\section{The Mellin transform and layer potential operators on cones} \label{sec:mellinoncones}
	
\subsection{Identification of $\MM K(z)$} \label{subsec:ident}
Let $\Gamma$ be a polyhedral cone. \red{In any study of the double layer potential on $\Gamma$, it is essential to analyze the Mellin transform $\MM K(z)$, as seen in \cite{Elschner, QN, Rat92}. An explicit identification of $\MM K(z)$ was made by Qiao and Nistor \cite{QN}, in terms of layer potentials for Schr\"odinger operators on the spherical polygon $\gamma$. We will therefore recall a number of their calculations.}
\red{We remind the reader that the single layer potential of $-\Delta$ on $\Gamma$ is given by 
	\begin{equation*}
	    \mathcal{S}u(x) =\frac{1}{4\pi}\int_{\Gamma}\frac{u(y)}{|x-y|} \, dS(y), \qquad {x\in\Gamma},  
	\end{equation*}
		 and that
 the direct value of the double layer potential of $u$ on $\Gamma$ is defined by
	\begin{align*}
	Ku(x)=-\frac{1}{4\pi}\int_{\Gamma} u(y)\frac{\partial}{\partial n_y}\frac{1}{|x-y|} \, dS(y)=\frac{1}{4\pi}\int_{\Gamma} u(y)\frac{(y-x)\cdot n_y}{|x-y|^3} \, dS(y), \qquad x\in\Gamma,
	\end{align*}
	 where $dS$ denotes the standard surface measure on $\Gamma$ and $n_y$ is the outward normal vector at a.e. $y \in \Gamma$.}

In spherical coordinates, $x=r\omega$ and $y=r' \omega'$, where $r=\dist(0,x)$, $r' = \dist(0, y)$, and $\omega, \omega' \in \gamma$, we have that
\begin{equation*} 
	\begin{aligned}
	\mathcal{S}u(r\omega)&=\frac{1}{4\pi}\int_{\R_+\times\gamma} \frac{r'}{|r\omega-r'\omega'|}u(r'\omega') \, d\omega' dr' \\
	&=\frac{1}{4\pi}\int_{\R_+\times\gamma} \frac{u(r'\omega')}{|(r/r')\omega-\omega'|} \,d\omega' {dr'}, \qquad x\in\Gamma,
	\end{aligned}
	\end{equation*}
	and
	\begin{equation*}
	\begin{aligned}
	Ku(r\omega)&=-\frac{1}{4\pi}\int_{\R_+\times\gamma} u(r'\omega')\frac{rr' \omega \cdot n_{\omega'}}{|r\omega-r'\omega'|^3} \, d\omega' dr' \\
	&=-\frac{1}{4\pi}\int_{\R_+\times\gamma} u(r'\omega')\frac{(r/r')\omega \cdot n_{\omega'}}{|(r/r')\omega-\omega'|^3} \, d\omega' \frac{dr'}{r'}, \qquad x\in\Gamma.
	\end{aligned}
	\end{equation*} 
	Therefore $\mathcal{S}_0:=M_{r^{-1/2}}\mathcal{S} M_{r^{-1/2}}$ is a Mellin convolution operator with operator-valued convolution kernel
	\begin{equation}\label{kernelS0}
	\mathcal{S}_0(t,\omega,\omega')=\frac{1}{4\pi}\frac{1}{t^{1/2} |t\omega-\omega'|},
	\end{equation}
	 while $K$ is itself a Mellin convolution operator with with kernel
	\begin{equation} \label{eq:Kismellinconv}
	K(t,\omega,\omega')=-\frac{1}{4\pi} \frac{t \omega \cdot n_{\omega'}}{|t \omega-\omega'|^3}=-\frac{1}{4\pi}\frac{t\omega \cdot n_{\omega'}}{(t^2-2t\omega\omega'+1)^{3/2}}.
	\end{equation}

	Let $\Phi:C_c^\infty(\hat{\Gamma})\rightarrow L^2_{\loc}(\hat{\Gamma})$ be the standard fundamental solution of $-\Delta$, understood as the operator defined by
	\begin{align*}
	\Phi u(x)&=\frac{1}{4\pi} \int_{\hat{\Gamma}} \frac{u(y)}{|x-y|} \, dy=\frac{1}{4\pi} \int_{\R_+ \times \hat{\gamma}}\frac{u(r'\omega')(r')^2 }{|r\omega-r'\omega'|} \, dS(\omega')dr'\\
	&=\frac{1}{4\pi} \int_{\R_+\times \hat{\gamma}}\frac{u(r'\omega')(r')^2 }{|\frac{r}{r'}\omega-\omega'|}\, dS(\omega')\frac{dr'}{r'},
	\end{align*}
	where $x=r\omega$, $y=r'\omega'$, $\omega, \omega' \in \hat{\gamma}$, and $dS$ is the surface measure on $\hat{\gamma}$.
	Consider $\Phi_0=M_{r^{-1}}\Phi M_{r^{-1}}$,
	\begin{equation*}
	\Phi_0 u(x)=\frac{1}{4\pi} \int_0^\infty \int_{\hat{\gamma}}\frac{u(r'\omega')}{\frac{r}{r'}|\frac{r}{r'}\omega-\omega'|}dS(\omega')\frac{dr'}{r'}, \qquad u\in C_c^\infty(\hat{\Gamma}).
	\end{equation*}
	Then $\Phi_0$ is a Mellin convolution operator with kernel
	\begin{equation} \label{eq:Phi0ismellinconv}
	\Phi_0(t,\omega,\omega')=\frac{1}{4\pi} \frac{1}{t|t\omega-\omega'|}.
	\end{equation}
	By this formula, $\MM\Phi_0(z) \colon C_c^\infty(\hat{\gamma}) \to L^2_{\loc}(\hat{\gamma})$ exists for $1 < \mre z < 2$, cf. \eqref{eq:PhiMellinasymp}, and
	$$(\MM \Phi_0 u)(z) = \MM\Phi_0(z) \MM u(z), \qquad u \in C_c^\infty(\hat{\Gamma}).$$
	
	On the other hand, the Laplacian in spherical coordinates is given by
	$$
	\Delta=\frac{1}{r^2}((r\partial_r)^2+r\partial_r+\Delta_{{\hat{\gamma}}}),
	$$
	where $\Delta_{\hat{\gamma}}$ denotes the restriction of the Laplace--Beltrami operator $\Delta_{{S^2}}$ to $\hat{\gamma}$,
    	$$\Delta_{\hat{\gamma}} f(\omega) = \Delta f(x/|x|)|_{x=\omega}, \qquad \omega \in S^2.$$
     \red{To understand the interaction between the Laplacian and the Mellin transform}, note, for $u \in C_c^\infty(\hat{\Gamma})$, that
\begin{equation} \label{eq:Minteract1}	
 \MM(r\partial_r u)(z) = -z \MM u(z),
 \end{equation}
	\red{and that the Mellin transform and the Laplace-Beltrami operator commute},
\begin{equation}\label{eq:Minteract2}	
 \MM(\Delta_{\hat{\gamma}} u)(z)  = \Delta_{\hat{\gamma}} \MM u(z).
 \end{equation}
	\begin{proposition}[\cite{QN}] \label{prop:fundsol}
		For $1 < \mre z < 2$, $\MM \Phi_0(z)$ is the fundamental solution for $-\Delta_{\hat{\gamma}} + 1/4 -(z-3/2)^2$, restricted to $\hat{\gamma}$.
	\end{proposition}
	\begin{proof}
		Since $-\Delta \Phi u=u$ for any $u \in C_c^\infty(\hat{\Gamma})$, we have that 
		$$
		u=-M_r\Delta\Phi M_{r^{-1}} u = -M_r\Delta M_r\Phi_0 u.
		$$
		Noting that $M_r\Delta M_r=(r\partial_r)^2+3 r\partial_r + 2 + \Delta_{\hat{\gamma}}$, we find, \red{applying \eqref{eq:Minteract1} and \eqref{eq:Minteract2}},that
		\begin{equation*}
		\mathcal{M}u(z)= -(z^2 - 3z + 2 + \Delta_{\hat{\gamma}})(\MM \Phi_0 u)(z) = (-\Delta_{\hat{\gamma}} + 1/4 -(z-3/2)^2 )(\MM \Phi_0 u)(z).
		\end{equation*}
		Since $(\MM \Phi_0 u)(z) = \MM\Phi_0(z) \MM u(z)$ it follows that (the kernel of) $\MM \Phi_0(z)$ is the fundamental solution for $-\Delta_{\hat{\gamma}} + 1/4 -(z-3/2)^2$. This fundamental solution is unique, by the positive definiteness of $-\Delta_{\hat{\gamma}}$ and the fact that $\mre \, (1/4 -(z-3/2)^2) > 0$.
	\end{proof}
	\red{We now recall one of the main results of \cite{QN}. See also \cite[Lemma~3.2]{Rat92}.}
	\begin{theorem}[\cite{QN}]
		For $-1/2 < \mre z < 1/2$, in terms of kernels, we have on $\gamma$ that
		$$
		\MM K(z+1/2)=-\partial_{n_{\omega'}}(\MM \Phi_0)(z+3/2).
		$$
		That is, $\MM K(z+1/2)$ is the direct value on $\gamma$ of the double layer potential operator of $-\Delta_{\hat{\gamma}} + 1/4 - z^2$.
	\end{theorem}
	
	\begin{proof}
		For $-3/2 < \mre z < 3/2$, \red{by \eqref{eq:Kismellinconv}}, the kernel of $\MM K(z+1/2)$ is given by
		$$
		\MM K(z+1/2)(\omega,\omega')=-\frac{1}{4\pi}  \int_0^\infty t^{z+1/2}\frac{t\omega \cdot n_{\omega'}}{|t \omega-\omega'|^3} \, \frac{dt}{t},
		$$
		where $\omega, \omega' \in \gamma$, $\omega \neq \omega'$. Furthermore, for $-1/2 < \mre z < 1/2$, \red{by \eqref{eq:Phi0ismellinconv} and direct calculations}, the normal derivative of the kernel of $(\MM \Phi_0)(z+3/2)$ is given by
		\begin{align*}
		-\partial_{n_{\omega'}}(\MM \Phi_0)(z+3/2)(\omega,\omega')&=-\int_0^\infty t^{z+3/2}\partial_{n_{\omega'}}\Phi_0(t,\omega,\omega')\,\frac{dt}{t}\\&=-\frac{1}{4\pi}\int_0^\infty t^{z+3/2}\frac{\omega \cdot n_{\omega'}}{|t\omega-\omega'|^3}\, \frac{dt}{t}.
		\end{align*}
		The result now follows from Proposition~\ref{prop:fundsol}.
	\end{proof}

\subsection{The adjoint of $\MM K(z)$ and the Kellogg argument on $\gamma$}
Let $H(z) = \MM K(z+1/2)$, $-3/2 < \mre z < 3/2$, and let $0 \leq \alpha < 1$. Recall from Section~\ref{subsec:mellinconv} that 
$K \colon L^2_\alpha(\Gamma) \to L^2_\alpha(\Gamma)$ is unitarily equivalent to 
\begin{equation} \label{eq:Ktensor}
I \otimes H(z) \colon L^2(\mre z =0)\otimes L^2_\alpha(\gamma) \to  L^2(\mre z =0)\otimes L^2_\alpha(\gamma).
\end{equation}
For $\mre z = 0$, we will in this section combine the description of $H(z)$ as the double layer potential operator of $-\Delta_{\hat{\gamma}} + 1/4 - z^2$ with a variant of an argument due to O.D. Kellogg, in order to show that every eigenvalue of $H^\ast(z) \colon L^2_{-\alpha}(\gamma) \to L^2_{-\alpha}$ is real.  The success of the Kellogg argument relies on the fact that the potential $1/4 - z^2 > 0$ for $\mre z = 0$.

In this discussion, $H^\ast(z) = (H(z))^\ast$ denotes the adjoint of $H(z)$ with respect to the $L^2(\gamma)$-pairing. For $f \in L^2_\alpha(\gamma)$ and $g \in L^2_{-\alpha}(\gamma)$, we see through the involution $t \mapsto 1/t$ that
\begin{align*}
    \langle H(z)f,g\rangle_{L^2(\gamma)}
    &=-\frac{1}{4\pi}\int_\gamma \int_\gamma \int_0^\infty t^{z+1/2}\frac{t\omega \cdot n_{\omega'}}{|t \omega-\omega'|^3} \, \frac{dt}{t}f(\omega') \, d\omega' \overline{g(\omega)} \, {d\omega}\\
    &=-\frac{1}{4\pi}\int_\gamma \int_\gamma \int_0^\infty t^{-z+1/2}\frac{t\omega \cdot n_{\omega'}}{| \omega-t\omega'|^3 } \, \frac{d t}{t}
\overline{g(\omega)} \, {d\omega} f(\omega') \, d\omega' \\
&= \langle f, \MM K^\dagger(-\overline{z}+1/2)g\rangle_{L^2(\gamma)},
\end{align*}
where $K^\dagger$ denotes the adjoint of $K$ with respect to the $L^2_0(\Gamma)$-pairing, which has operator-valued convolution kernel
$$K^\dagger(t, \omega, \omega') = - \frac{1}{4\pi t}\frac{(1/t)\omega' \cdot n_\omega}{|(1/t)\omega' - \omega|^3} = - \frac{1}{4\pi} \frac{t \omega' \cdot n_\omega}{|\omega' - t\omega|^3}.$$
Therefore,
$$H^\ast(z) = \MM K^\dagger(1/2 - \bar{z}), \qquad -3/2 < \mre z < 3/2.$$
Applying the argument of Section~\ref{subsec:mellinconv} again, we conclude the following.
\begin{lemma} \label{lem:Kdagger}
For $0 \leq \alpha < 1$, $K^\dagger \colon L^2_{-\alpha}(\Gamma) \to L^2_{-\alpha}(\Gamma)$ is unitarily equivalent to
$$I \otimes H^\ast(z) \colon L^2(\mre z =0)\otimes L^2_{-\alpha}(\gamma) \to  L^2(\mre z =0)\otimes L^2_{-\alpha}(\gamma).$$
\end{lemma}
One could of course have arrived at this lemma directly by taking the adjoint in \eqref{eq:Ktensor}, but the preceding calculations are instructive for later arguments.

We now give the promised Kellogg argument, which relies on having access to a fairly complete layer potential theory of $-\Delta_{\hat{\gamma}}+1/4 - z^2$ on the Lipschitz domain $\hat{\gamma} \subset S^2$. For this purpose, we refer to the (much more general) theory of boundary layer potentials for Lipschitz domains in Riemannian manifolds \cite{MT99, MT01}, developed by M. Mitrea and M. Taylor. 
	\begin{lemma}\label{eigenvH*}
		Suppose that $0 \leq \alpha < 1$, $\mre z = 0$, and $\lambda \not\in (-1/2,1/2)$. Then the operator $\lambda I-H^\ast(z):L^2_{-\alpha}(\gamma)\longrightarrow L^2_{-\alpha}(\gamma)$ is injective.
	\end{lemma}
	\begin{proof}
		Suppose that $0\neq f\in L^2_{-\alpha}(\gamma)$ satisfies $(\lambda I-H^\ast(z))f=0$ for some $\lambda \in \C$. Let $V=1/4 - z^2 > 0$, and let $\tilde{\Sg}(z)$ denote the single layer potential on $\gamma$ for $-\Delta_{\hat{\gamma}} + V$. By Proposition~\ref{prop:fundsol}, for suitable functions $g$ on $\gamma$,
		$$\tilde{\Sg}(z)g(\omega) = \int_{\gamma} \left[\MM \Phi_0(z+3/2)\right](\omega, \omega') g(\omega') \, d\omega', \qquad \omega \in S^2.$$
By H\"older's inequality, $L^2_{-\alpha}(\gamma)$ is continuously contained in $L^p(\gamma) = L^p(\gamma, d\omega)$ for every $1 < p<\frac{2}{\alpha +1}$. Therefore, by the results of \cite{MT99}, $H^\ast(z) \colon L^2_{-\alpha}(\gamma) \to L^p(\gamma)$ is bounded, and, for $g \in L^2_{-\alpha}(\gamma)$,
		\begin{equation*}
		\left(\frac{\partial}{\partial n} \tilde{\Sg}(z) g\right)_{\pm}(\omega) =\left(\pm \frac{1}{2} I - H^\ast(z)\right) g(\omega),  \qquad \textrm{a.e. } \omega \in \gamma.
		\end{equation*}
		Here $\left(\frac{\partial}{\partial n} \tilde{\Sg}(z) g\right)_{\pm}$ denotes the $\partial_{n_{\omega}}$-derivative of $\tilde{\Sg}(z)g$,  defined in terms of  non-tangential limits from inside $\hat{\gamma}$ or from its exterior $\hat{\gamma}_- = S^2 \setminus \overline{\hat{\gamma}}$ (corresponding to the $\pm$ in the notation). For any $g \in {C^\infty_{\arc}(\gamma)}$, we have that $(-\Delta_{\hat{\gamma}} + V)\tilde{\Sg}(z) g(\omega)=0$, $\omega\not\in \gamma$, and applying Green's formula, justified with maximal function estimates as in \cite[Proposition~4.1]{MT99}, yields that
		$$
		I_+(g) :=\int_{\hat{\gamma}} |\nabla \tilde{\Sg}(z) g(\omega)|^2 + V |\tilde{\Sg}(z) g(\omega)|^2 \, dS(\omega)=\int_\gamma \overline{\tilde{\Sg}(z) g(\omega)}\left(\ \frac{\partial}{\partial n} \tilde{\Sg}(z) g \right)_+(\omega) \, d \omega
		$$
		and
		$$
		I_-(g) :=\int_{\hat{\gamma}_-} |\nabla \tilde{\Sg}(z) g(\omega)|^2 + V |\tilde{\Sg}(z) g(\omega)|^2 \, dS(\omega) =-\int_\gamma \overline{\tilde{\Sg}(z) g(\omega)}\left( \frac{\partial}{\partial n} \tilde{\Sg}(z) g \right)_-(\omega) \, d\omega,
		$$
		where, in this proof only, $\nabla$ denotes the gradient of $S^2$. Note here that $\tilde{\Sg}(z)$ maps $L^p(\gamma)$ into the Sobolev space $W^{1,p}(\gamma)$ \cite[Eq. (7.49)]{MT99}, and thus in particular that $\tilde{\Sg}(z)$ is bounded as a map $\tilde{\Sg}(z) \colon L^2_{-\alpha}(\gamma) \to L^q(\gamma)$, $\frac{1}{p} + \frac{1}{q} = 1$. Hence the energies $I_\pm(g)$ depend continuously on $g \in L^2_{-\alpha}(\gamma)$. By Fatou's lemma, it thus also follows that $L^2_{-\alpha}(\gamma) \ni g \mapsto \tilde{\Sg}(z) g \in L^2(S^2, dS(\omega))$ and 
		$$L^2_{-\alpha}(\gamma) \ni g \mapsto \nabla \tilde{\Sg}(z) g \in L^2(S^2, dS(\omega)) \oplus L^2(S^2, dS(\omega))$$
		 are continuous. 
		
		Therefore the equations
		\begin{equation*}
		I_\pm(g)= \pm \int_\gamma \overline{\tilde{\Sg}(z) g(\omega)}\left( \pm\frac{1}{2} I - H^\ast(z)\right) g(\omega) \, d\omega
		\end{equation*}
		remain valid for general $g \in L^2_{-\alpha}(\gamma)$, and in particular for $g=f$. Moreover, since $f$ is an eigenfunction of $H^\ast(z)$, we have
		\begin{equation*}
		I_\pm(f)= \left(\frac{1}{2} \mp \lambda \right)\int_{\gamma} \overline{\tilde{\Sg}(z) f(\omega)}f(\omega) \, d\omega,
		\end{equation*}
		and thus
		\begin{equation} \label{eq:kelloggen}
		2\lambda(I_+(f) + I_-(f))= I_-(f)-I_+(f).
		\end{equation}
		
		Both of the energies $I_\pm(f)$ must be positive. For if not, we would have that $\tilde{S}(z)f \equiv 0$ in either $\hat{\gamma}$ or $\hat{\gamma}_-$, and therefore $\tilde{S}(z)f= 0$ on $\gamma$ by \cite[Proposition~3.8]{MT99}. This would yield that $f = 0$, since $\tilde{S}(z)$ is injective on $L^p(\gamma)$ \cite[Equation (1.20)]{MT01}. Therefore \eqref{eq:kelloggen} implies that $\lambda \in (-1/2,1/2)$.
	\end{proof}

\begin{remark} \label{rmk:energynorm} \rm
The proof in particular shows that for $\mre z = 0$ and $0 \neq g \in L^2(\gamma)$, 
$$\langle \tilde{S}(z) g, g \rangle_{L^2(\gamma)} = \int_{S^2} |\nabla \tilde{\Sg}(z) g(\omega)|^2 + V |\tilde{\Sg}(z) g(\omega)|^2 \, dS(\omega) > 0.$$
This is the basis for constructing the energy space $\mathcal{E}(\gamma, -\Delta_{\hat{\gamma}} + 1/4 - z^2) \simeq H^{-1/2}(\gamma)$ in Section~\ref{sec:energycone}. Furthermore, once constructed, the proof of Lemma~\ref{eigenvH*} shows that $L^2_{-\alpha}(\gamma)$ is continuously contained in $\mathcal{E}(\gamma, -\Delta_{\hat{\gamma}} + 1/4 - z^2)$, $0 \leq \alpha < 1$. Dualizing, this can be interpreted as a non-sharp fractional Hardy inequality:
$$\int_{\gamma} |f|^2 q^{-\alpha} \, d\omega \leq C_\alpha \|f\|_{H^{1/2}(\gamma)}^2, \qquad 0 \leq \alpha < 1.$$
\end{remark}
	
\section{Spectral theory on $L^2_\alpha(\partial \Omega)$} \label{sec:L2a}
\subsection{Analysis of $H(z)$}
Let $\Gamma$ be a polyhedral cone and let $K$ be its Neumann--Poincar\'e operator. Recall that we write $H(z) = \MM K(z+1/2)$, $-3/2<\mre z<3/2$, so that  
$$
H(z)v(\omega)= -\frac{1}{4\pi} \int_0^\infty \int_\gamma t^{z+1/2}\frac{t \omega \cdot n_{\omega'}}{|t \omega-\omega'|^3}v(\omega') \, d\omega'\frac{dt}{t}, \quad v\in C^\infty_{\arc}(\gamma), \;\omega\in\gamma\setminus\{E_1,\dots,E_J\}.
$$
Observe that $H(z)$ is pointwise well-defined, since $\omega \cdot n_{\omega'} =0$ whenever $\omega$ and $\omega'$ belong to the same arc $\gamma_j$. For each corner $E_j$ of $\gamma$, we choose a function $\varphi_j \in C^{\infty}_{\arc}(\gamma)$ such that $0\le \varphi_j \le 1$ on $\gamma$, $\varphi_j$ is supported in a small neighbourhood of $E_j$, and $\varphi_j \equiv 1$ close to $E_j$. We then introduce the decomposition $H(z)=H_0(z)+H_1(z)$, where 
\begin{equation} \label{eq:Hdecomp}
H_0(z)=\sum_{1\le j\le J} \varphi_j H(z)\varphi_j, \quad H_1(z)=H(z)-H_0(z).
\end{equation}
The starting point of this section is the following result of Elschner. We have extracted a slightly more precise statement than given in \cite[Theorem~2.1]{Elschner}, which follows from its proof. We denote the interior angle made by $\gamma$ at $E_j$ by $\beta_j$. 
	
	\begin{lemma}[\cite{Elschner}]\label{PropertiesH} Let $0\le \alpha<1$, $\epsilon>0$, and $\delta > 0$. Then \begin{itemize}
			\item[i)] The operator-valued map $z \mapsto H(z) \colon L^2_\alpha(\gamma) \to L^2_\alpha(\gamma)$ is analytic in the strip $-3/2<\mre z<3/2$. 
			\item[ii)] If the supports $\supp \varphi_j$ are chosen sufficiently small, $1 \leq j \leq J$, the decomposition \eqref{eq:Hdecomp} satisfies the following on the closed strip $-3/2+\delta\le\mre z\le3/2-\delta$: 
	 $$\|H_0(z)\|_{\LL(L^2_\alpha(\gamma))}\le \frac{1 + \varepsilon}{2}\max_{1\le j\le J}\left|\frac{\sin((\pi-\beta_j)(1-\alpha)/2)}{\sin(\pi(1-\alpha)/2)} \right|,
		$$
			  and $H_1(z) \colon L^2_\alpha(\gamma) \to L^2_\alpha(\gamma)$ is Hilbert--Schmidt with 
			  $$\lim_{ |\mim z| \to \infty} \|H_1(z)\|_{\mathcal{S}_2(L^2_\alpha(\gamma))} = 0,$$
			  \red{where $\mathcal{S}_2(L^2_\alpha(\gamma))$ denotes the Hilbert--Schmidt  norm on $L^2_\alpha(\gamma)$}.
		\end{itemize}
	\end{lemma}

Next we will describe the spectrum of $H(z) \colon L^2_\alpha(\gamma) \to L^2_\alpha(\gamma)$. For $1 \leq j \leq J$, let
$$ \Sigma_{\alpha, \beta_j} = \left \{\frac{1}{2}\frac{\sin((\pi-\beta_j)(\frac{1-\alpha}{2}+i\xi))}{\sin(\pi(\frac{1-\alpha}{2}+i\xi))} \, : \, -\infty \leq \xi \leq \infty \right\}.$$
This is a simple closed curve in $\mathbb{C}$ with $0 \in \Sigma_{\alpha, \beta_j}$, described in detail in \cite[Lemma~12]{Per21}, \red{see some examples in Figure \ref{sigmapi2}.} Let $\widehat{\Sigma}_{\alpha, \beta_j}$ denote $\Sigma_{\alpha, \beta_j}$ together with its interior, and let $\widehat{\Sigma}_{\alpha, \beta_j}^{-} = -\widehat{\Sigma}_{\alpha, \beta_j}$ denote its reflection in the imaginary axis.
	\begin{lemma}\label{spectH}			
	Let $0\le \alpha<1$ and $-3/2<\mre z<3/2$. Then
	$$\sigma \left(H(z), L^2_\alpha(\gamma)\right)=\bigcup_{{1\le j\le J   }}(\widehat{\Sigma}_{\alpha, \beta_j} \cup \widehat{ \Sigma}_{\alpha, \beta_j}^-)\cup \Lambda_{\gamma, z}^\alpha,$$
	where $\Lambda_{\gamma, z}^\alpha$ is a countable set of isolated eigenvalues in the complement of $\cup_j (\widehat{\Sigma}_{\alpha, \beta_j} \cup \widehat{ \Sigma}_{\alpha, \beta_j}^-)$.	Moreover, each isolated eigenvalue $\lambda_z \in \Lambda_{\gamma, z}^\alpha$ depends continuously on $z$, and if $\mre z=0$, then $\Lambda_{\gamma, z}^\alpha \subset (-1/2,1/2)$. 
	\end{lemma}
\begin{proof} 		
	Let $\chi_i$ be the characteristic function of $\bar{\gamma}_i$, $i = 1,2$, and let $\varphi_{1,1}=\chi_1\varphi_1$, $\varphi_{1,2}=\chi_2\varphi_1$. We first analyze the operator $\varphi_{1,2}H(z)\varphi_{1,1}$. Without loss of generality we may assume that $E_1$ is the north pole of $S^2$, that $\gamma_1$ lies in the plane $x_2 = 0$, and that $n_{\omega'}=(0,-1,0)$ for $\omega' \in \gamma_1$.  Then $\beta_1$ is the polar angle between $\gamma_1$ and $\gamma_2$. Therefore, $\omega \in \gamma_2$ and $\omega' \in \gamma_1$ can be written 
	$\omega=(\cos \beta_{1} \sin s, \sin\beta_{1} \sin s , \cos s)$, $\omega'=(\sin s', 0, \cos s')$, where, as before, $s$ and $s'$ denote the arc lengths along $\gamma$ from $E_1$ to $\omega$ and $\omega'$, respectively. 
	
	In this parametrization of $\omega \in \gamma_2$ and $\omega' \in \gamma_1$, the kernel of $H(z)$ can be written
	\begin{align*}
	H(z)(s,s') = \frac{1}{4\pi} \int_0^\infty t^{z+3/2}\frac{ b}{(t^2-2at+1)^{3/2}}\frac{dt}{t},
	\end{align*}
	where  $a= a(s,s') = \omega\cdot\omega'=\cos s\cos s'+\cos\beta_{1} \sin s \sin s'$ and $b= b(s,s') = - n_{\omega'}\cdot\omega= \sin\beta_{1} \sin s$. From \eqref{eq:KMellinasymp}, we thus have that
	$$\varphi_{1,2}H(z)\varphi_{1,1}(\omega,\omega') = \frac{1}{4\pi}\frac{b(s,s')}{1-a(s,s')} + I_1(s,s'), \qquad \omega \in \gamma_2, \; \omega' \in \gamma_1.$$
	where the kernel $I_1(s,s') = O(s (1+|\log(s+s')|))$ defines an operator $I_1 \colon L^2_\alpha(\gamma) \to L^2_\alpha(\gamma)$ which is Hilbert--Schmidt. We introduce the notation
	$$Y_{\beta_1}(s,s') = \frac{1}{4\pi}\frac{b(s,s')}{1-a(s,s')}.$$
	
	For a small number $s_0 > 0$ (reflecting the size of the support of $\varphi_1$), we now consider $\varphi_{1,2} Y_{\beta_1} \varphi_{1,1}$, naturally understood as an operator on $L^2( (0,s_0), \, s^{-\alpha} ds)$. Performing the change of variables $\sigma=\tan (s/2)$, which induces an isomorphism
	$$Q \colon L^2( (0,s_0), \, s^{-\alpha}ds) \to L^2((0, \tan(s_0/2)), \sigma^{-\alpha}d\sigma), \quad Q v(\sigma) = v(2\arctan \sigma),$$
	we obtain
	\begin{equation} \label{eq:changeofvar}
	\begin{aligned}
	\Upsilon_{\beta_1} u(\sigma') &:= Q\varphi_{1,2} Y_{\beta_1} \varphi_{1,1}Q^{-1} u(\sigma') \\ &= \frac{\sin\beta_{1}}{2\pi} Q\varphi_{1,2}(\sigma) \int_0^{\tan (\frac{s_0}{2})}\frac{\frac{\sigma}{\sigma'}}{1+\left(\frac{\sigma}{\sigma'}\right)^2-2\cos\beta_{1}\frac{\sigma}{\sigma'}} Q\varphi_{1,1}(\sigma') u(\sigma')\,\frac{d\sigma'}{\sigma'}. 
	\end{aligned}
	\end{equation}
	We see that $\Upsilon_{\beta_1}$ coincides with the localization $Q\varphi_{1,2} Z_{\beta_1} Q\varphi_{1,1}$ of a Mellin convolution operator $Z_{\beta_1} \colon L^2(\mathbb{R}_+, \sigma^{-\alpha} d\sigma) \to L^2(\mathbb{R}_+, \sigma^{-\alpha} d\sigma)$ with convolution kernel
	$$Z_{\beta_1}(t) = \frac{\sin \beta_1}{2\pi} \frac{t}{t^2 - 2t \cos\beta_1 + 1}.$$
	This is the same convolution kernel that appears in the study of the Neumann--Poincar\'e operator for planar polygonal domains \cite{Mitreapolygon}. We further conjugate with the unitary multiplication operator
	$$M_{\sigma^{\frac{1-\alpha}{2}}}:L^2((0, \tan(s_0/2)), \sigma^{-\alpha} \, d\sigma)\rightarrow L^2\left((0, \tan(s_0/2)), d\sigma/\sigma \right),$$
	obtaining that
	$$M_{\sigma^{\frac{1-\alpha}{2}}} \Upsilon_{\beta_1} M_{\sigma^{\frac{\alpha-1}{2}}} \colon L^2\left((0, \tan(s_0/2)), d\sigma/\sigma \right) \to L^2\left((0, \tan(s_0/2)), d\sigma/\sigma \right)$$
	coincides with the localization $Q\varphi_{1,2} Z_{\alpha, \beta_1} Q\varphi_{1,1}$ of the Mellin convolution operator $Z_{\alpha, \beta_1}$ with convolution kernel
	$$Z_{\alpha, \beta_1}(t) = \frac{\sin \beta_1}{2\pi} \frac{t^{\frac{3-\alpha}{2}}}{t^2 - 2t \cos\beta_1 + 1}.$$
	
	Therefore $M_{\sigma^{\frac{1-\alpha}{2}}} \Upsilon_{\beta_1} M_{\sigma^{\frac{\alpha-1}{2}}}$ belongs to the algebras of Mellin operators considered in \cite{Elschner1987, Lewis1990, Mitreapolygon}. \red{By applying the result described in Section~\ref{subsec:mellinpdo}}, we obtain that
	 \begin{equation*}
	\sigma_{\ess}\left(\varphi_{1,2} Y_{\beta_1} \varphi_{1,1},L^2_\alpha(\gamma)\right) = \left\{\MM Z_{\alpha, \beta_1}(i\xi) \, : \, -\infty \leq \xi \leq \infty \right\} = \Sigma_{\alpha, \beta_1}.
	\end{equation*}
	Furthermore, for any $\lambda \notin \Sigma_{\alpha, \beta_1}$, we have for the Fredholm index that
	$$
	\ind (\varphi_{1,2} Y_{\beta_1} \varphi_{1,1} -\lambda, L^2_\alpha(\gamma))=W(\lambda, \Sigma_{\alpha, \beta_1}),
	$$
	where $W(\lambda, \Sigma_{\alpha, \beta_1})$ denotes the winding number of $\lambda$ with respect to the Jordan curve $\Sigma_{\alpha, \beta_1}$. In particular,
 $$\widehat{\Sigma}_{\alpha, \beta_1} \subset \sigma\left(\varphi_{1,2} Y_{\beta_1} \varphi_{1,1}, L^2_\alpha(\gamma)\right).$$
	
By geometrical symmetry, as operators on $L^2( (0,s_0), \, s^{-\alpha}ds)$, $\varphi_{1,1} H(z)\varphi_{1,2}$ only differs from $\varphi_{1,2} H(z)\varphi_{1,1}$ by a compact operator. More precisely, with respect to the decomposition $L^2_\alpha(\gamma) = L^2_{\alpha}(\gamma_1)\oplus L^2_{\alpha}(\gamma_2) \oplus\cdots\oplus L^2_{\alpha}(\gamma_J)$, we have that
	\begin{equation}\label{decompH}
\varphi_1H(z)\varphi_1=\begin{pmatrix} 0 & \varphi_{1,2} Y_{\beta_1} \varphi_{1,1} &\cdots & 0\\  \varphi_{1,2} Y_{\beta_1} \varphi_{1,1} & 0 & \cdots & 0 \\ \vdots &\vdots &\vdots &\vdots \\0 & 0 & \cdots & 0 \end{pmatrix} + \textrm{compact}.
\end{equation}
The previous considerations of essential spectrum and Fredholm index, and some elementary linear algebra, therefore yield that
$$\widehat{\Sigma}_{\alpha, \beta_1} \cup \widehat{ \Sigma}_{\alpha, \beta_1}^- \subset \sigma(\varphi_1H(z)\varphi_1, L^2_\alpha(\gamma)),$$
and that 
$$\ind ( \varphi_1H(z)\varphi_1 - \lambda, L^2_\alpha(\gamma)) = W(\lambda, \Sigma_{\alpha, \beta_1} \cup \Sigma_{\alpha, \beta_1}^-), \qquad \lambda \notin \Sigma_{\alpha, \beta_1} \cup \Sigma_{\alpha, \beta_1}^-.$$ 

The preceding analysis applies equally well to any of the operators $\varphi_j H(z) \varphi_j$, $j = 2,\ldots, J$. Adding up and using the compactness of $H_1(z)$ we therefore find that
$$\bigcup_{1 \leq j \leq J} (\widehat{\Sigma}_{\alpha, \beta_j} \cup \widehat{ \Sigma}_{\alpha, \beta_j}^-) \subset \sigma(H(z), L^2_\alpha(\gamma)),$$
and that 
$$\ind ( H(z) - \lambda, L^2_\alpha(\gamma)) = 0, \qquad \lambda \notin \bigcup_{1 \leq j \leq J} (\widehat{\Sigma}_{\alpha, \beta_j} \cup \widehat{ \Sigma}_{\alpha, \beta_j}^-).$$
It follows from the considerations in \cite[Lemma~12]{Per21} that the complement of $\bigcup_{j} (\widehat{\Sigma}_{\alpha, \beta_j} \cup \widehat{ \Sigma}_{\alpha, \beta_j}^-)$ is connected. Thus the analytic Fredholm theorem yields that the spectrum $\Lambda_{\gamma, z}^\alpha$ in this complement consists of isolated eigenvalues. If $\mre z = 0$, Lemma~\ref{eigenvH*} implies that $\Lambda_{\gamma,z}^\alpha \subset (-1/2,1/2)$, since the index is $0$.

	Finally, we shall see that any eigenvalue $\lambda_z \in \Lambda_{\gamma, z}^\alpha$ depends continuously on $z$ in the strip $-3/2<\mre z<3/2$. Consider the disc $C_r:=\{\lambda\in\C: \: |\lambda-\lambda_z| \leq r\}$, for $r > 0$ so small that 
	$$\sigma(H(z), L^2_\alpha(\gamma)) \cap C_r =\{\lambda_z\}.$$
	By the analyticity of $z' \mapsto H(z')$, we know that $z' \mapsto (H(z') - \lambda)^{-1}$ is analytic for $z'$ close to $z$ and $\lambda \in \partial C_r$. In particular, 
	$$\lim_{z'\to z}P_{\partial_{C_r}}^{H(z')} = P_{\partial_{C_r}}^{H(z)},$$
	where
	$$P_{\partial_{C_r}}^{H(z)}=-\frac{1}{2\pi i}\int_{\partial C_r}(H(z)-\lambda)^{-1} \, d\lambda$$
	denotes the spectral projection corresponding to $\partial C_r$. 	
	\end{proof}
	We shall also require the following lemma in our analysis.
	\begin{lemma} \label{lem:eigdec}
	Under the conditions of Lemma~\ref{spectH}, the sets of isolated eigenvalues $\Lambda^{\alpha}_{\gamma, z}$ are increasing in $0 \leq \alpha < 1$. 
	\end{lemma}
\begin{proof}
	Suppose that $\lambda \in \Lambda^{\alpha}_{\gamma, z}$. Then $\ind (H(z) - \lambda, L^2_\alpha(\gamma)) = 0$ by (the proof of) Lemma~\ref{spectH}, and thus $\bar{\lambda}$ is an eigenvalue of $H^\ast(z) \colon L^2_{-\alpha'}(\gamma) \to L^2_{-\alpha'}(\gamma)$ for $\alpha \leq \alpha' < 1$, since the spaces $L^2_{-\alpha'}(\gamma)$ are increasing in $\alpha'$. However, the regions $\widehat{\Sigma}_{\alpha, \beta_j}$ are decreasing in $0 \leq \alpha < 1$ \cite[Lemma~12]{Per21}, and thus also $\ind (H(z) - \lambda, L^2_{\alpha'}(\gamma)) = 0$. Therefore $\lambda \in \Lambda^{\alpha'}_{\gamma, z}$, $\alpha \leq \alpha' < 1$. 
\end{proof}
\begin{remark} \rm \label{rmk:shrink}
	Since the spaces $L^2_\alpha(\gamma)$ are decreasing in $\alpha$, it is obvious that the entire point spectrum of $H(z) \colon L^2_\alpha(\gamma) \to L^2_\alpha(\gamma)$ is decreasing. Intuitively, more isolated eigenvalues of $H(z) \colon L^2_\alpha(\gamma) \to L^2_\alpha(\gamma)$ are ``uncovered'' as $\alpha$ increases, as the other part of the spectrum in Lemma~\ref{spectH} shrinks.
	\begin{figure}[h!]
  \includegraphics[scale=0.4]{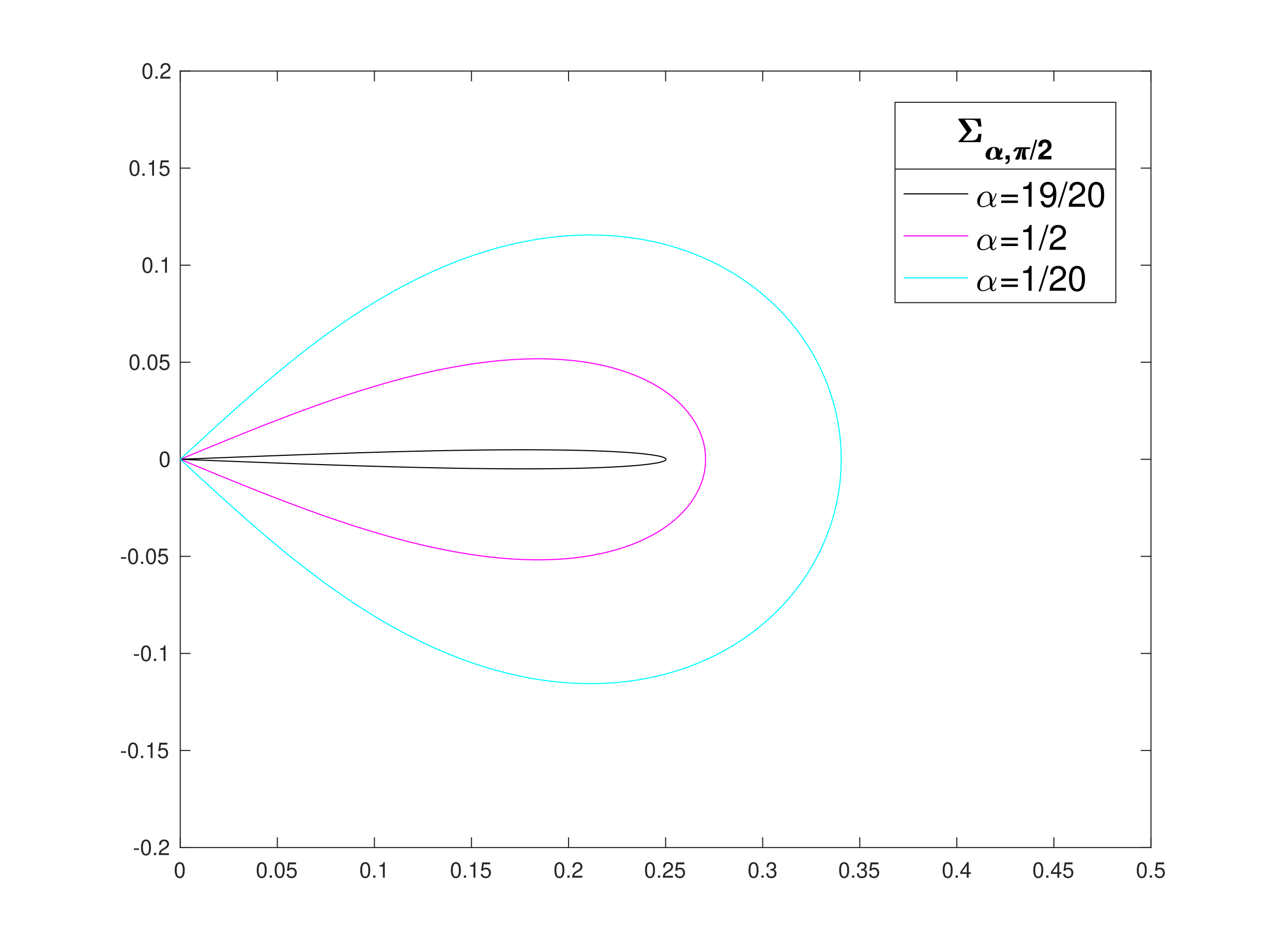}
  \caption{The Jordan curves $\Sigma_{\alpha,\pi/2}$, for different values of $\alpha.$}
  \label{sigmapi2}
\end{figure}
\end{remark}


\subsection{Analysis on polyhedral cones}
As in the previous subsection, $K$ denotes the Neumann--Poincar\'e operator of a polyhedral cone $\Gamma$. We will now investigate the spectrum of $K \colon L^2_\alpha(\Gamma) \to L^2_\alpha(\Gamma)$, $0 \leq \alpha < 1$, starting with the following lemma. Recall that $K^\dagger$ denotes the adjoint of $K$ with respect to the $L^2_0(\Gamma)$-pairing.
\begin{lemma} \label{lem:Weyl} Let $0 \leq \alpha < 1$.
	Suppose that we are given $\xi \in \mathbb{R}$, $\lambda \in \mathbb{C}$, and $d > 0$. Then for any $g \in L^2_\alpha(\gamma)$ with $\|g\|_{L^2_\alpha(\gamma)} = 1$ and $\varepsilon > 0$, there exists $w \in L^2_\alpha(\Gamma)$ with 
	$$\supp w \subset [0,d] \times \gamma, \quad  \|w\|_{L^2_\alpha(\Gamma)} = 1,$$
	such that
	\begin{equation} \label{eq:bumpest1}
	\| (K - \lambda) w\|_{L^2_\alpha(\Gamma)}^2 \leq 4\|(H(i\xi) - \lambda) g\|_{L^2_\alpha(\gamma)}^2 + \varepsilon^2.
	\end{equation}
	Similarly, for any $\tilde{g} \in L^2_{-\alpha}(\gamma)$ with $\|\tilde{g}\|_{L^2_{-\alpha}(\gamma)} = 1$ and $\varepsilon > 0$, there exists $\tilde{w} \in L^2_{-\alpha}(\Gamma)$ with $\supp \tilde{w} \subset [0,d] \times \gamma$, $\|\tilde{w}\|_{L^2_{-\alpha}(\Gamma)} = 1$, and
	\begin{equation} \label{eq:bumpest2}
	\| (K^\dagger - \bar{\lambda}) \tilde{w}\|_{L^2_{-\alpha}(\Gamma)}^2 \leq 4\|(H^\ast(i\xi) - \bar{\lambda}) \tilde{g}\|_{L^2_{-\alpha}(\gamma)}^2 + \varepsilon^2.
	\end{equation}
\end{lemma}
\begin{proof}
	For some small $0 < A < 1$ to be specified later, let 
	$$f(r) = {\frac{1}{\sqrt{\frac{1}{2} \log (1/A)}}} \chi_{[A, \sqrt{A}]}(r) r^{-1/2 - i\xi}, \qquad r > 0,$$ 
	and
	$$w(r\omega) = f(r) g(\omega), \qquad r > 0, \; \omega \in \gamma.$$
	Then $\supp w \subset  {[A, \sqrt{A}]}\times \gamma$ and $\|w\|_{L^2_\alpha(\Gamma)} = 1$. Furthermore,
	$$h(i\eta) := \MM (M_{r^{1/2}} f)(i\eta) = {\frac{-i}{\sqrt{\frac{1}{2} \log (1/A)}}} \frac{A^{\frac{i}{2} (\eta - \xi)} - A^{i(\eta - \xi)}}{\eta - \xi}, \qquad \eta \in \R,$$
	so that
	\begin{equation} \label{eq:lochest}
	|h(i\eta)| \leq { \frac{2}{\sqrt{\frac{1}{2} \log (1/A)}\;|\eta - \xi|}, \qquad \eta \in \mathbb{R}. }
	\end{equation}
	By continuity, choose $\delta > 0$ so that $\|H(i\eta) - H(i\xi)\| \leq \varepsilon/3$ whenever $|\eta - \xi| < \delta$. 
	Consider now the fact that
	\begin{align*}
	\| (K - \lambda) w\|_{L^2_\alpha(\Gamma)}^2 &= \frac{1}{2\pi} \|I \otimes (H(z) -\lambda) (h \otimes g) \|^2_{L^2(\mre z =0)\otimes L^2_\alpha(\gamma)} \\ &= \frac{1}{2\pi} \int_{\mathbb{R}} |h(i\eta)|^2 \|(H(i\eta) - \lambda) g\|^2_{L^2_\alpha(\gamma)} \, d\eta. 
	\end{align*}
	For $|\eta - \xi| < \delta$ we have that
	\begin{multline*}
	\frac{1}{2\pi} \int_{|\eta - \xi| < \delta} |h(i\eta)|^2 \|(H(i\eta) - \lambda) g\|^2_{L^2_\alpha(\gamma)} \, d\eta \\ \leq \frac{2}{\pi} \int_{|\eta - \xi| < \delta} |h(i\eta)|^2 (\|(H(i\xi) - \lambda) g\|^2_{L^2_\alpha(\gamma)} + \varepsilon^2/9) \, d\eta \leq  4 \|(H(i\xi) - \lambda) g\|^2_{L^2_\alpha(\gamma)} + \varepsilon^2/2,
	\end{multline*}
	where we used that $\frac{1}{2\pi} \int_{\mathbb{R}} |h(i \eta)|^2 \, d\eta = 1$.
	By the uniform boundedness of $H(i\eta)$ (Lemma~\ref{PropertiesH}) and \eqref{eq:lochest}, we clearly have for $|\eta - \xi| > \delta$ that
	$$\lim_{A \to 0^+} \frac{1}{2\pi} \int_{|\eta - \xi| > \delta} |h(i\eta)|^2 \|(H(i\eta) - \lambda) g\|^2_{L^2_\alpha(\gamma)} \, d\eta = 0.$$
	This yields \eqref{eq:bumpest1}, if we choose $A$ sufficiently small. 
	
	The inequality \eqref{eq:bumpest2} is established through the same reasoning, after recalling Lemma~\ref{lem:Kdagger}.
\end{proof}
We can now establish the following invertibility result.
\begin{theorem}\label{specK}
	Let $0 \leq \alpha < 1$ and $\lambda\in\C$. Then $K - \lambda\colon  L^2_{\alpha}(\Gamma) \to L^2_{\alpha}(\Gamma)$ is invertible if and only if $\lambda\not\in \sigma(H(z), L^2_\alpha(\gamma))$ for every $z$ with $\mre z = 0$, and 
	\begin{equation} \label{eq:uniforminv}
	\sup_{\mre z = 0} \|(H(z)-\lambda)^{-1}\|_{\mathcal{L}(L^2_\alpha(\gamma))} < \infty.
	\end{equation}
	
	Let $d > 0$. If $K - \lambda$ is not invertible, then there is a sequence $(w_n)_{n=1}^\infty$ with $\supp w_n \subset [0, d] \times \gamma$ that is either a singular Weyl sequence for $K-\lambda$ or $K^\dagger - \bar{\lambda}$. That is, in the former case, $\|w_n\|_{L^2_{\alpha}(\Gamma)} = 1$ for every $n$, $w_n \to 0$ weakly in $L^2_\alpha(\Gamma)$, and $\|(K-\lambda)w_n\|_{L^2_\alpha(\Gamma)} \to 0$. In the latter case, $\|w_n\|_{L^2_{-\alpha}(\Gamma)} = 1$, $w_n \to 0$ weakly in $L^2_{-\alpha}(\Gamma)$, and $\|(K^\dagger-\bar{\lambda})w_n\|_{L^2_{-\alpha}(\Gamma)} \to 0$. In particular,
	$$\sigma(K, L^2_{\alpha}(\Gamma)) = \sigma_{\ess}(K, L^2_{\alpha}(\Gamma)).$$
\end{theorem}
\begin{proof}
	Assume first that $H(z) - \lambda$ is invertible for every $\mre z=0$ and that \eqref{eq:uniforminv} holds. Then, by Lemma~\ref{lem:tensorbdd}, $I\otimes(H(z) - \lambda)^{-1}$ defines a bounded operator on $L^2(\mre z=0)\otimes L^2_\alpha(\gamma)$ which is the inverse of $I \otimes (H(z) - \lambda)$. Since this latter operator is unitarily equivalent to 
	$K - \lambda: L^2_{\alpha}(\Gamma) \to L^2_{\alpha}(\Gamma)$, this shows that $K - \lambda$ is invertible.
	
	For the converse, suppose first that $H(z) - \lambda$ is invertible for every $z$ on the line $\mre z=0$, but that $\|(H(z) - \lambda)^{-1}\|_{\mathcal{L}(L^2_\alpha(\gamma))}$ is unbounded on $\mre z=0$. Then there exists a sequence $(\xi_n) \subset \mathbb{R}$ and functions $g_n$ such that $\|g_n\|_{L^2_\alpha(\gamma)} = 1$ and $\lim_{n \to \infty} \|(H(i\xi_n) - \lambda) g_n\|=0$. Applying Lemma~\ref{lem:Weyl}, we construct a unit sequence $(w_n)$ with $\supp w_n \subset [0, 2^{-n}] \times \gamma$ and $\lim_{n \to \infty} \|(K-\lambda)w_n\|_{L^2_\alpha(\Gamma)} = 0.$ In particular, $K - \lambda$ is not invertible in this case.
	
	Finally, we have to consider the case when there is a $\xi_0 \in \mathbb{R}$ such that $\lambda \in \sigma(H(i\xi_0), L^2_\alpha( \gamma))$. If $H(i\xi_0) - \lambda$ is not bounded from below, we construct the desired Weyl sequence for $K-\lambda$ from Lemma~\ref{lem:Weyl}. If $H(i\xi_0) - \lambda$ is bounded from below (but not invertible), it must be that $\bar{\lambda}$ is an eigenvalue of $H^\ast(i\xi_0)$, and we use Lemma~\ref{lem:Weyl} to construct a Weyl sequence for $K^\dagger - \bar{\lambda}$.
\end{proof}
As a consequence we obtain the following incomplete description of the spectrum of $K \colon L^2_\alpha(\Gamma) \to L^2_\alpha(\Gamma)$, illustrated in Figure~\ref{fig:conespecL2}. By the proof of Lemma~\ref{spectH}, the Jordan curve $\Sigma_{\alpha, \beta_j}$ arises as the Mellin transform of either a positive or a negative kernel, depending on the sign of $\pi - \beta_j$. We therefore have that
$$|\Sigma_{\alpha, \beta_j}| := \max\{|\lambda| \, : \, \lambda \in \Sigma_{\alpha, \beta_j}\} = \frac{1}{2}\left|\frac{\sin((\pi-\beta_j)(1/2-\alpha/2))}{\sin(\pi(1/2-\alpha/2))} \right|,$$
and
$$\widehat{\Sigma}_{\alpha, \beta_j} \cup \widehat{ \Sigma}_{\alpha, \beta_j}^- \subset C_{|\Sigma_{\alpha, \beta_j}|} := \{\lambda \, : \, |\lambda| \leq |\Sigma_{\alpha, \beta_j}|\}.$$
A calculus argument shows that $|\Sigma_{\alpha, \beta_j}|$ is decreasing in $0 \leq \alpha < 1$, for every $j$ \cite[Lemma~12]{Per21}. If $c > d$ we use the convention that $(c, d] = \emptyset$. 
	\begin{theorem}\label{contessspecK}
		Let $\Gamma$ be a polyhedral cone with angles $\beta_1, \ldots, \beta_J$, and let $0\le \alpha<1$. We let $j^\ast$ be an index such that 
		$$|\Sigma_{\alpha, \beta_{j^*}}| = \max_j |\Sigma_{\alpha, \beta_j}|,$$
		and denote {\rm
		$$\Lambda^\alpha = \{\lambda \, : \, \lambda \textrm{ is an isolated eigenvalue of } H(z) \colon L^2_\alpha(\gamma) \to L^2_\alpha(\gamma), \textrm{ for some} \mre z = 0\}.$$ }
		Then there are two numbers $0 \leq \mu_{\pm} < 1/2$, independent of $\alpha$, such that 
		$$\Lambda^\alpha = \left[-\mu_-, -|\Sigma_{\alpha, \beta_{j^*}}|\right) \cup \left(|\Sigma_{\alpha, \beta_{j^*}}|, \mu_+\right].$$
		Furthermore, we have that $\sigma(K, L^2_{\alpha}(\Gamma)) = \sigma_{\ess}(K, L^2_{\alpha}(\Gamma))$ and that 
		\begin{equation}\label{conteinspecK}
		\bigcup_{{1\le j\le J   }}(\widehat{\Sigma}_{\alpha, \beta_j} \cup \widehat{ \Sigma}_{\alpha, \beta_j}^-) \cup \Lambda^\alpha 
		\subset \sigma \left(K  , L^2_\alpha(\Gamma) \right) \subset C_{|\Sigma_{\alpha, \beta_{j*}}|} \cup \Lambda^\alpha.\end{equation}
	\end{theorem}

\begin{figure}[ht]
	\centering
	\includegraphics[width=0.95\linewidth]{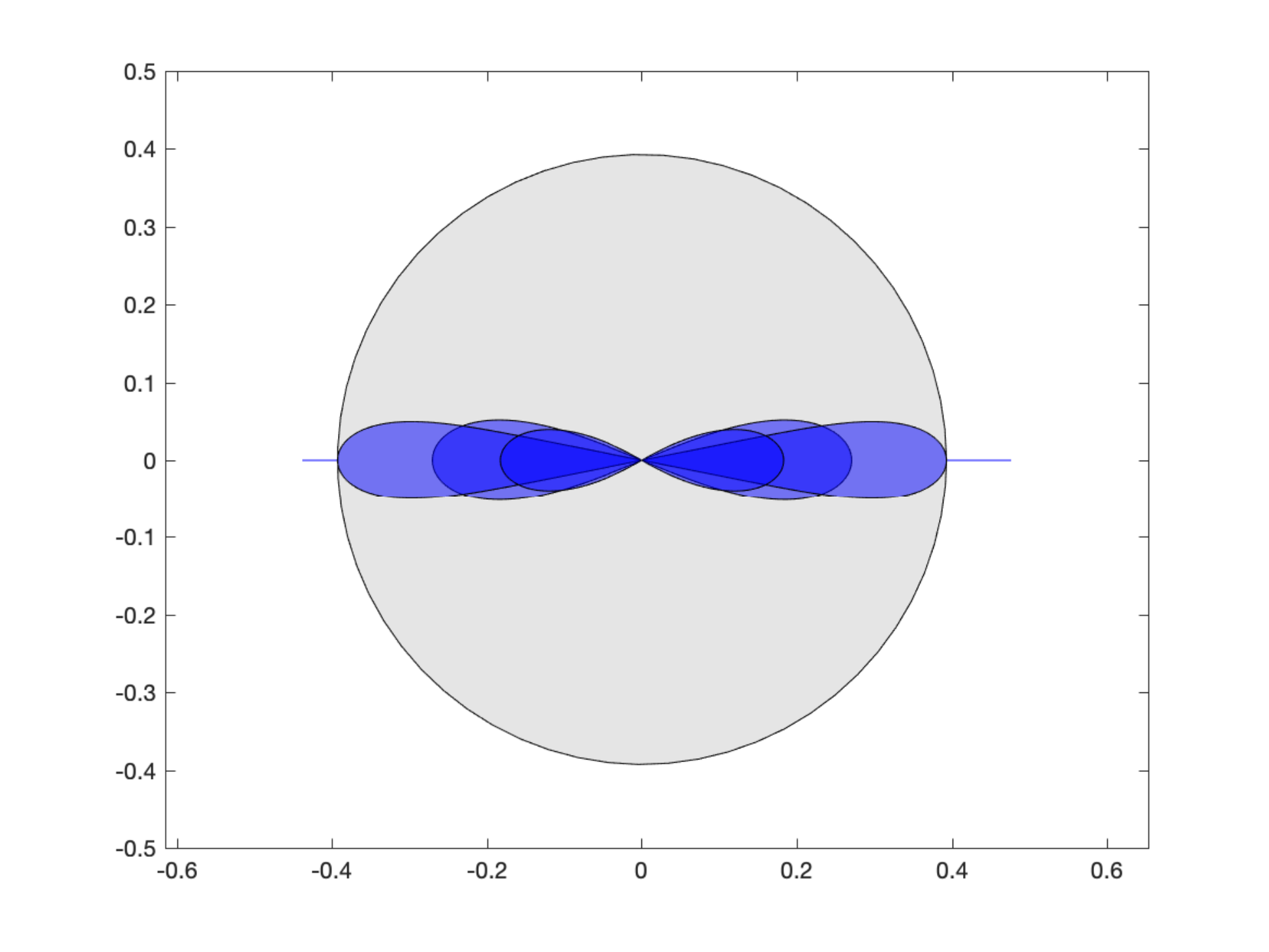}
	\caption{An illustration of Theorem~\ref{contessspecK} for a polyhedral cone $\Gamma$ with angles $\beta_1 = \pi/4$, $\beta_2 = \pi/2$, and $\beta_3 = 2\pi/3$, and $\alpha = 1/2$. The set 
		$$\Theta = \bigcup_{1 \leq j \leq 3}(\widehat{\Sigma}_{1/2, \beta_j} \cup \widehat{ \Sigma}_{1/2, \beta_j}^-) \cup \Lambda^{1/2}$$
	has been plotted in the complex plane, in various shades of blue. The set $\Lambda^{1/2}$ is generically drawn; one or both of the intervals comprising this set may actually be empty. The region $C_{|\Sigma_{1/2, \beta_{1}}|} \setminus \Theta$ is shaded in gray.}
	\label{fig:conespecL2}
\end{figure}

	\begin{remark} \label{rmk:resolvent} \rm
		The interval $[-|\Sigma_{\alpha, \beta_j}|, |\Sigma_{\alpha, \beta_j}|]$ is contained in $\widehat{\Sigma}_{\alpha, \beta_j} \cup \widehat{ \Sigma}_{\alpha, \beta_j}^-$. Therefore Theorem~\ref{contessspecK} characterizes all real points in $\sigma \left(K  , L^2_\alpha(\Gamma) \right)$. There is a gap of complex points $\lambda \in C_{|\Sigma_{\alpha, \beta_{j*}}|} \setminus \bigcup_{{1\le j\le J   }}(\widehat{\Sigma}_{\alpha, \beta_j} \cup \widehat{ \Sigma}_{\alpha, \beta_j}^-)$ in \eqref{conteinspecK} because while $H(z) - \lambda$ is invertible for every $\mre z = 0$, we do not know how to control the resolvent $(H(z) - \lambda)^{-1}$ uniformly in $z$, cf. Theorem~\ref{specK}.
	\end{remark}
	\begin{proof}
	By Lemma~\ref{spectH}, 
	$$\Lambda^\alpha = \bigcup_{\mre z = 0} \Lambda_{\gamma,z}^\alpha \subset (-1/2,1/2) \setminus [-|\Sigma_{\alpha, \beta_{j^\ast}}|, |\Sigma_{\alpha, \beta_{j^\ast}}|].$$ 
	Let 
	$$\mu_{\pm}^\alpha = \sup_{\mre z = 0} \max \{ \lambda \, : \, \pm \lambda \in \Lambda_{\gamma,z}^\alpha\}.$$
	Suppose that $\mu_+^\alpha > |\Sigma_{\alpha, \beta_{j^\ast}}|$, so that there is $z$ such that 
	$$\lambda_{z, +} := \max \Lambda_{\gamma, z}^\alpha > |\Sigma_{\alpha, \beta_{j^\ast}}|.$$
	By Lemma~\ref{spectH}, $\lambda_{z, +}$ depends continuously on $z$ as long as $\lambda_{z, +} > |\Sigma_{\alpha, \beta_{j^\ast}}|$, tracing out an interval as $z$ varies. From item ii) of Lemma~\ref{PropertiesH}, we know that
	\begin{equation} \label{eq:essrad}
	\varlimsup_{|\xi|\to \infty }\|H(i\xi)\|_{\mathcal{L}(L^2_\alpha(\gamma))}\le|\Sigma_{\alpha, \beta_{j^\ast}}|.
	\end{equation}
	It therefore follows that there is a point $z^\ast$ for which $\lambda_{z^\ast, +} = \mu_+^\alpha$ and that 
	$$\Lambda^\alpha \cap [0,1/2) = \left(|\Sigma_{\alpha, \beta_{j^\ast}}|, \mu_+^\alpha\right].$$
	Note also that if $0 \leq \alpha, \alpha' < 1$ are two numbers such that $\mu_+^\alpha > |\Sigma_{\alpha, \beta_{j^\ast}}|$ and $\mu_+^{\alpha'} > |\Sigma_{\alpha', \beta_{j^\ast}}|$, then $\mu_+^\alpha = \mu_+^{\alpha'}$, by Lemma~\ref{lem:eigdec} and Remark~\ref{rmk:shrink}.
	
	If $\mu_-^\alpha > |\Sigma_{\alpha, \beta_{j^\ast}}|$, a similar argument applies for $\Lambda^\alpha \cap (-1/2, 0]$.
	
	By Lemma~\ref{spectH} and Theorem~\ref{specK}, we now have that
	$$\bigcup_{{1\le j\le J   }}(\widehat{\Sigma}_{\alpha, \beta_j} \cup \widehat{ \Sigma}_{\alpha, \beta_j}^-) \cup \Lambda^\alpha 
	\subset \sigma \left(K  , L^2_\alpha(\Gamma) \right).$$
	On the other hand, if $\lambda \notin C_{|\Sigma_{\alpha, \beta_{j*}}|} \cup \Lambda^\alpha$,  \eqref{eq:essrad} allows us to apply a Neumann series argument and the continuity of $H(z)$ to see that
	$$\sup_{\mre z = 0} \|(H(z) - \lambda)^{-1}\|_{\mathcal{L}(L^2_\alpha(\gamma))} < \infty.$$
	Therefore $\lambda \notin \sigma \left(K  , L^2_\alpha(\Gamma) \right)$ in this case, by Theorem~\ref{specK}.
		\end{proof}
	
	\subsection{Localization to $\partial \Omega$}
	Let $\partial \Omega$ be a polyhedron and let $K$ be its Neumann--Poincar\'e operator. In this subsection we will show that the essential spectrum $\sigma_{\ess}(K, L^2_\alpha(\partial \Omega))$, $0 \leq \alpha < 1$, is determined by the tangent polyhedral cones $\Gamma_i$ of the vertices of $\partial \Omega$, $i = 1, \ldots, I$.	We first need the following lemma.
	\begin{lemma}
\label{commcompact}
Let $0 \leq \alpha < 1$, and let $\varphi$ be a Lipschitz function on $\partial \Omega$. Then the commutator $[K,\varphi]=K\varphi-\varphi K$ is compact on $L^2_\alpha(\partial\Omega)$.
\end{lemma}
\begin{proof}
By a partition of unity it is sufficient to prove the statement for $\psi_1 [K,\varphi] \psi_2$, with a polyhedral cone $\Gamma$ in place of $\partial \Omega$ and $\psi_1$, $\psi_2$, $\varphi$ being compactly supported Lipschitz functions on $\Gamma$. Then
$$\psi_1 [K,\varphi] \psi_2 f(r\omega) = \int_{\Gamma} C(r\omega, r'\omega') f(r'\omega') \, dr' \, q(\omega')^{-\alpha} d\omega',$$
where the kernel $C$ is supported in $[0, A] \times \gamma \times [0, A] \times \gamma$ for some $0 < A < \infty$ and satisfies
$$| C(r\omega, r'\omega') | \lesssim \frac{r' q(\omega')^\alpha }{|r\omega- r'\omega'|}.$$
This immediately implies that the integral operator with kernel $C(r\omega, r'\omega') \chi_{\{|r\omega- r'\omega'| > \varepsilon\}}$ is Hilbert-Schmidt on $L^2_\alpha(\Gamma)$ for every $\varepsilon > 0$,
$$\int_{|r\omega- r'\omega'| > \varepsilon} |C(r\omega, r'\omega')|^2 \, dr \, q(\omega)^{-\alpha} d\omega \, dr' \, q(\omega')^{-\alpha} d\omega' < \infty. $$
For $|r\omega- r'\omega'| < \varepsilon$ we have that
$$| C(r\omega, r'\omega') | \lesssim \varepsilon^{1/2} \underbrace{\frac{r' q(\omega')^\alpha }{|r\omega- r'\omega'|^{3/2}}}_{T(r\omega, r'\omega')}.$$
It is thus sufficient to prove that the integral operator $T$ with kernel $T(r\omega, r'\omega')$ is bounded on $L^2([0,A] \times \gamma, dr \, q(\omega)^{-\alpha} d\omega)$. By applying the unitary transformation
$$M_{r^{\frac{1}{2}}} M_{q(\omega)^{\frac{1-\alpha}{2}}} \colon L^2([0,A] \times \gamma, dr \, q(\omega)^{-\alpha} d\omega) \to L^2([0,A] \times \gamma, r^{-1} dr \, q(\omega)^{-1} d\omega),$$
 it is equivalent to prove that the integral operator $\widetilde{T}$,
$$\widetilde{T} f(r\omega) = \int_{\gamma} \int_0^A \widetilde{T}(r\omega, r'\omega') f(r' \omega') \, \frac{dr'}{r'} \, \frac{d\omega'}{q(\omega')},$$
with kernel
$$\widetilde{T}(r\omega, r'\omega') = \left(\frac{r}{r'}\right)^{\frac{1}{2}} \left(\frac{q(\omega)}{q(\omega')}\right)^{\frac{1-\alpha}{2}} \frac{(r')^2 q(\omega') }{|r\omega- r'\omega'|^{3/2}},$$
is bounded on $L^2([0,A] \times \gamma, r^{-1} dr \, q(\omega)^{-1} d\omega)$. To do this we verify that $\widetilde{T}$ is bounded on $L^1 = L^1([0,A] \times \gamma, r^{-1} dr \, q(\omega)^{-1} d\omega)$ and $L^\infty$ and apply the Riesz-Thorin interpolation theorem.

To see that it is bounded on $L^1$, it is sufficient to see that 
$$\sup_{r', \omega'} \int_{\gamma} \int_0^A \widetilde{T}(r\omega, r'\omega') \, \frac{dr}{r} \, \frac{d\omega}{q(\omega)} < \infty.$$
By the change of variable $h = r/r'$ we have that
$$\int_{\gamma} \int_0^A \widetilde{T}(r\omega, r'\omega') \, \frac{dr}{r} \, \frac{d\omega}{q(\omega)} \leq A^{1/2} \int_{\gamma} \left(\frac{q(\omega)}{q(\omega')}\right)^{\frac{1-\alpha}{2}} q(\omega') \int_0^\infty h^{\frac{1}{2}} \frac{1}{|h \omega- \omega'|^{3/2}}  \, \frac{dh}{h} \, \frac{d\omega}{q(\omega)},$$
where
$$\int_0^\infty h^{\frac{1}{2}} \frac{1}{|h \omega- \omega'|^{3/2}}  \, \frac{dh}{h} = \int_0^\infty h^{\frac{1}{2}} \frac{1}{(h^2 - 2h \omega \cdot \omega' + 1)^{3/4}}  \, \frac{dh}{h} \lesssim \frac{1}{(1-\omega \cdot \omega')^{\frac{1}{4}}} = \frac{2^{1/4}}{|\omega - \omega'|^{1/2}}.$$
Thus we are left to show that
\begin{equation} \label{eq:gammafinite}
\sup_{\omega'} \int_{\gamma} \left(\frac{q(\omega)}{q(\omega')}\right)^{\frac{1-\alpha}{2}} \frac{q(\omega')}{|\omega - \omega'|^{1/2}} \, \frac{d\omega}{q(\omega)} < \infty.
\end{equation}
Since $0 \leq \alpha < 1$, we are only concerned with the situation that both $\omega$ and $\omega'$ are close to the same corner of $\gamma$, and we then introduce arc-length parametrization. Thus we have to verify that
$$\sup_{0 < s' < 1} \int_{0}^1 \left(\frac{s}{s'}\right)^{\frac{1-\alpha}{2}} \frac{s'}{|s-s'|^{1/2}} \, \frac{ds}{s} < \infty$$
and
$$\sup_{0 < s' < 1} \int_{0}^1 \left(\frac{s}{s'}\right)^{\frac{1-\alpha}{2}} \frac{s'}{(s+s')^{1/2}} \, \frac{ds}{s} < \infty,$$
corresponding to whether $\omega$ and $\omega'$ lie on the same arc or not. By the change of variable $t = s/s'$, we have
$$ \int_{0}^1 \left(\frac{s}{s'}\right)^{\frac{1-\alpha}{2}} \frac{s'}{|s\pm s'|^{1/2}} \, \frac{ds}{s} \leq (s')^{\frac{1}{2}}\int_{0}^{1/s'} t^{\frac{1-\alpha}{2}} \frac{1}{|t \pm 1|^{1/2}} \, \frac{dt}{t} \lesssim (s')^{\frac{1}{2}}(1 + |\log s'|),$$
where we used that $0 \leq \alpha < 1$. We conclude that \eqref{eq:gammafinite} holds.

The boundedness of $\widetilde{T}$ on $L^\infty$ can be proved similarly.
\end{proof}

We now present our localization theorem for $L^2_\alpha(\partial \Omega)$.
	\begin{theorem}\label{locarg1}
Let $0 \leq \alpha < 1$ and let $K$ be the Neumann--Poincar\'e operator of a polyhedron $\partial \Omega$. For each vertex of $\partial \Omega$, let $K_i$ denote the Neumann--Poincar\'e operator of the corresponding tangent polyhedral cone $\Gamma_i$, $i = 1, \ldots, I$.  

Then, for $\lambda \in \mathbb{C}$, $K - \lambda$ is Fredholm on $L^2_\alpha(\partial\Omega)$ if and only if $K_i - \lambda$ is invertible on $L^2_\alpha(\Gamma_i)$ for every $i=1,\dots,I$. That is,
$$\sigma_{\ess}(K, L^2_\alpha(\partial\Omega)) = \bigcup_{1 \leq i \leq I} \sigma(K_i, L^2_\alpha(\Gamma_i)).$$
The spectra $\sigma(K_i, L^2_\alpha(\Gamma_i))$ have been described in Theorem~\ref{contessspecK}.
	\end{theorem}
	\begin{remark} \rm
	One would like to accompany this theorem with a Kellogg-type argument for $K \colon L^2_\alpha(\partial \Omega) \to L^2_\alpha(\partial \Omega)$, cf. Lemma~\ref{eigenvH*}, but we have been unable to produce such a result.
	\end{remark}
	\begin{proof}
		 Let $\{\varphi_i\}_{i=1}^I$ be a partition of unity of $\partial \Omega$, as described in Section~\ref{subsec:weightedL2}, and let $\{\eta_i\}_{1\le i\le I}$ be a family of Lipschitz functions on $\partial\Omega$ such that $\eta_i\equiv 1$ in a neighbourhood of $\supp \varphi_i$ and $\eta_i \equiv 0$ on $\partial \Omega \setminus \Gamma_i$. As usual, by very slight abuse of notation, we understand functions such as $\varphi_i$ and $\eta_i$ as functions on both $\partial \Omega$ and $\Gamma_i$. 
		
	    In one direction, we argue along the lines of \cite[Lemma~1]{Mitrea99}. Suppose that $\lambda I-K_i$ is invertible on $L^2_\alpha(\Gamma_i)$, for every $i=1,\dots,I.$ Then, for $f\in L^2_\alpha(\partial\Omega)$,
	    \begin{align*}
	       \|f\|_{L^2_\alpha(\partial\Omega)}&\simeq\sum_{1\le i\le I}\|\varphi_i f\|_{L^2_\alpha(\Gamma_i)}\lesssim \sum_{1\le i\le I}\|(\lambda I- K_i)\varphi_i f\|_{L^2_\alpha(\Gamma_i)}\\
	       &\lesssim \sum_{1\le i\le I}\|{\eta_i}(\lambda I- K)\varphi_i f\|_{L^2_\alpha(\Gamma_i)}+  \sum_{1\le i\le I}\|{(1-\eta_i)}(\lambda I- K_i)\varphi_i f\|_{L^2_\alpha(\Gamma_i)}\\
	      	       &\lesssim  \sum_{1\le i\le I}\|(\lambda I- K)\varphi_i f\|_{L^2_\alpha(\partial\Omega)}+ \sum_{1\le i\le I}\|{(1-\eta_i)} K_i\varphi_i f\|_{L^2_\alpha(\Gamma_i)}.
	    \end{align*}
	   Since $(1-\eta_i)$ and $\varphi_i$ have disjoint supports, the operator ${(1-\eta_i)}K_i\varphi_i:{L^2_\alpha(\partial\Omega)}\rightarrow L^2_\alpha(\Gamma_i)$ is Hilbert--Schmidt. Furthermore, by Lemma \ref{commcompact}, the commutator $[K,\varphi_i]=K\varphi_i-\varphi_iK:{L^2_\alpha(\partial\Omega)}\rightarrow {L^2_\alpha(\partial\Omega)}$ is compact, for each $i=1,\dots,I.$ Therefore, there is a Hilbert space $\mathcal{H}$ and a compact operator $C \colon {L^2_\alpha(\partial\Omega)}\rightarrow \mathcal{H}$ such that
		    \begin{align*}
	       \|f\|_{L^2_\alpha(\partial\Omega)}&\lesssim \sum_{1\le i\le I}\|\varphi_i(\lambda I- K) f\|_{L^2_\alpha(\partial\Omega)}+ \|C f\|_{\mathcal{H}}\\
	       &\lesssim \|(\lambda I- K) f\|_{L^2_\alpha(\partial\Omega)} + \|C f\|_{\mathcal{H}}.
	    \end{align*}
	   This means that $\lambda I-K$ is upper semi-Fredholm on $L^2_\alpha(\partial\Omega)$. Applying the exact same argument to the $L^2_0(\partial \Omega)$-adjoint $K^\dagger$ shows that also $\bar{\lambda} I - K^\dagger$ is upper semi-Fredholm on $(L^2_\alpha(\partial\Omega))' \simeq L^2_{-\alpha}(\partial\Omega)$. Therefore $\lambda I - K$ is Fredholm.
	   
 For the converse, suppose that  there exists  $i_0\in \{1,\dots,I\}$ such that $\lambda I-K_{i_0}$ is not invertible on $L^2_\alpha(\Gamma_{i_0})$. By Theorem~\ref{specK}, there is a Weyl sequence $(w_n)$ for either $\lambda I-K_{i_0}$ or $\bar{\lambda} I-K_{i_0}^\dagger$, and we can choose the functions $w_n$ to have arbitarily small support around the vertex of $\Gamma_{i_0}$. We treat the former case; the argument for the second case is identical. 
 
 As a sequence $(w_n) \subset L^2_\alpha(\partial \Omega)$, $\|w_n\|_{L^2(\partial \Omega)} = 1$ for every $n$ and $w_n \to 0$ weakly. Furthermore, assuming we chose $w_n$ to have sufficiently small support,
$$(\lambda I - K)w_n = \eta_{i_0} (\lambda I - K) \varphi_{i_0} w_n - (1-\eta_{i_0}) K \varphi_{i_0} w_n, \; \textrm{ on } \partial \Omega.$$
Here $ (1-\eta_{i_0}) K\varphi_{i_0} w_n \to 0$ in norm as $n \to \infty$, since $(1-\eta_{i_0}) K \varphi_{i_0}$ is compact on $L^2_\alpha(\partial \Omega)$ and $w_n \to 0$ weakly, while
$$\eta_{i_0} (\lambda I - K ) \varphi_{i_0}w_n = \eta_{i_0} (\lambda I- K_{i_0} )w_n \to 0$$ 
because $(w_n)$ is a Weyl sequence for $(\lambda I - K_{i_0})$. Thus $(w_n)$ is a Weyl sequence also for $\lambda I - K$, and therefore $\lambda I - K$ is not Fredholm.  
	\end{proof}


	\section{Spectral theory on the energy space} \label{sec:energy}
	\subsection{The energy space of a polyhedral cone} \label{sec:energycone}	
	Let $\Gamma$ be a Lipschitz polyhedral cone. In this subsection we study the energy space $\mathcal{E}(\Gamma)$ defined in Section~\ref{subsec:energyspace}. 
	
	In Section~\ref{subsec:ident} we introduced the Mellin operator $\Sg_0 = M_{r^{-1/2}} \Sg M_{r^{-1/2}}$ on $\Gamma$. Comparing its convolution kernel \eqref{kernelS0} with Proposition~\ref{prop:fundsol}, we observe, for $\xi\in\R$, that $\MM{\Sg_0}(i\xi+1)$
	coincides with the single layer potential on $\gamma$ for $-\Delta_{\hat{\gamma}} + 1/4 + \xi^2$,
	$$\MM{\Sg_0}(i\xi+1) = \tilde{S}(i\xi).$$
	We therefore introduce the space $H^{-1/2}_\xi({\gamma}) = \mathcal{E}(\gamma, -\Delta_{\hat{\gamma}} + 1/4 + \xi^2)$ as the completion of $L^2(\gamma)$ in the norm
	$$\|g\|_{H^{-1/2}_\xi({\gamma})}^2 = \langle \tilde{S}(i\xi) g, g \rangle_{L^2(\gamma)},$$
	recalling from Remark~\ref{rmk:energynorm} that 
	\begin{equation} \label{eq:Sxinorm}
	\|g\|_{H^{-1/2}_\xi({\gamma})}^2 = \int_{S^2} |\nabla \tilde{\Sg}(i\xi) g(\omega)|^2 + (1/4 + \xi^2) |\tilde{\Sg}(i\xi) g(\omega)|^2 \, dS(\omega).
	\end{equation}
	
	The connection with $\mathcal{E}(\Gamma)$ arises from the computation
\begin{align*}\langle \Sg f,g\rangle_{L^2(\Gamma)}&= \langle M_{r} \Sg_0 M_{r^{-1}}M_{r^{3/2}}f, M_{r^{3/2}} g \rangle_{L^2( \mathbb{R}_+,\frac{dr}{r}) \otimes L^2(\gamma)} \\
&= \frac{1}{2\pi}\int_{\R} \langle \tilde{S}(i\xi) \mathcal{M}f(i\xi + 3/2), \mathcal{M}g(i\xi+3/2) \rangle_{L^2(\gamma)} \, d\xi \\
&= \frac{1}{2\pi} \int_{\R} \langle \mathcal{M}f(i\xi + 3/2), \mathcal{M}g(i\xi+3/2) \rangle_{H^{-1/2}_\xi({\gamma})} \, d\xi,
\end{align*}
which is not hard to justify for $f,g \in C^\infty_c( \cup_j F_j).$ By a density argument, $\frac{1}{\sqrt{2\pi}}\mathcal{M} M_{r^{3/2}}$ therefore induces a unitary map
$$\frac{1}{\sqrt{2\pi}}\mathcal{M} M_{r^{3/2}} \colon \mathcal{E}(\Gamma) \to L^2(\mathbb{R}, d\xi) \otimes \mathcal{E}(\gamma, -\Delta_{\hat{\gamma}} + 1/4 + \xi^2).$$

We now focus our attention on the spaces $\mathcal{E}(\gamma, -\Delta_{\hat{\gamma}} + 1/4 + \xi^2)$. Being a single layer potential operator, we already know that $\tilde{S}(i\xi) \colon L^2(\gamma) \to H^1(\gamma)$ is an isomorphism for every $\xi \in \R$ \cite[Proposition 7.5]{MT99}. We first establish that this isomorphism depends continuously on $\xi$.
\begin{lemma} \label{lem:Sxicont}
	The map 
	$$\R \ni \xi \mapsto \tilde{S}(i\xi) \in \mathcal{L}(L^2(\gamma), H^1(\gamma))$$
	is uniformly continuous.
\end{lemma}
\begin{proof}
	Recall from Section~\ref{subsec:ident} that the integral kernel of $\tilde{S}(i\xi)$ is given by
	\begin{equation*}
	(\tilde{S}(i\xi))(\omega, \omega') = (\MM \mathcal{S}_0(i\xi+1))(\omega, \omega')=\frac{1}{4\pi}\int_0^\infty \frac{t^{i\xi+1/2}}{|t\omega-\omega'|}\frac{dt}{t}, \qquad \omega, \omega' \in \gamma.
	\end{equation*}
	For $\xi, \xi' \in \R$, we have that
	\begin{align*}
	\left| \int_0^\infty \frac{t^{i\xi + 3/2} - t^{i\xi' + 3/2}}{(t^2 - 2t\omega \cdot \omega' + 1)^{3/2}} \, \frac{dt}{t} \right| &\leq |\xi - \xi'| \int_0^\infty \frac{t^{3/2}|\log t|}{(t^2 - 2t\omega \cdot \omega' + 1)^{3/2}} \, \frac{dt}{t} \\
	&\lesssim \frac{|\xi - \xi'|}{\sqrt{1 - \omega \cdot \omega'}} = \sqrt{2}\frac{|\xi - \xi'|}{|\omega - \omega'|}, \qquad \omega, \omega' \in \gamma.
	\end{align*} 
	Consequently, differentiating in the arc-length parameter $s$, $\omega = \omega(s)$, we find, for $\omega \in \gamma \setminus \{E_1, \ldots, E_J\}$ and $\omega' \neq \omega$, that
	$$|\partial_s (\tilde{S}(i\xi) - \tilde{S}(i\xi'))(\omega, \omega')| \lesssim |\xi - \xi'| \frac{|\partial_s \omega \cdot \omega'|}{|\omega - \omega'|} \lesssim |\xi - \xi'|.$$
	Through similar but simpler reasoning, the kernel itself satisfies the same estimate,
	$$(\tilde{S}(i\xi) - \tilde{S}(i\xi'))(\omega, \omega') \lesssim |\xi - \xi'|.$$
	It follows that
	\begin{equation*}
	\|\tilde{S}(i\xi) - \tilde{S}(i\xi')\|_{\mathcal{L}(L^2(\gamma), H^1(\gamma))} \lesssim |\xi - \xi'|. \qedhere
	\end{equation*}
\end{proof}

$\tilde{S}(i\xi) \colon L^2(\gamma) \to H^1(\gamma)$ is an isomorphism and $\tilde{S}(i\xi) \colon L^2(\gamma) \to L^2(\gamma)$ is a non-negative operator, by \eqref{eq:Sxinorm}. Duality therefore yields that $\tilde{S}(i\xi)\colon H^{-1}(\gamma)\to L^2(\gamma)$ is an isomorphism, and thus that
$$\langle \tilde{S}(i\xi)^2 u, u \rangle_{L^2(\gamma)} = \|\tilde{S}(i\xi) u\|_{L^2(\gamma)}^2 \simeq \|u \|_{H^{-1}(\gamma)}^2,$$
with implied constants depending on $\xi$. By interpolation, $\tilde{S}(i\xi) \colon H^{-1/2}(\gamma) \to H^{1/2}(\gamma)$ is also an isomorphism. Furthermore, interpolation yields that
$$\langle\tilde{S}(i\xi) u, u \rangle_{L^2(\gamma)} = \|\tilde{S}(i\xi)^{\frac{1}{2}} u \|^2_{L^2(\gamma)} \simeq \|u \|_{H^{-1/2}(\gamma)}^2,$$
initially for $u \in L^2(\gamma)$, see the proof of Theorem 15.1 in \cite{Lions-Magenes}. In other words,
$$\mathcal{E}(\gamma, -\Delta_{\hat{\gamma}} + 1/4 + \xi^2) \simeq H^{-1/2}(\gamma).$$
Lemma~\ref{lem:Sxicont} implies that this identification depends continuously on $\xi$.

\begin{lemma}
Let $\xi \in \R$. For any $\varepsilon > 0$, there is a $\delta > 0$ such that if $|\xi - \xi'| < \delta$, then
	$$ \left| \|u\|_{H^{-1/2}_{\xi}(\gamma)}-\|u\|_{H^{-1/2}_{\xi'}(\gamma)}\right| \leq \varepsilon \|u\|_{H^{-1/2}(\gamma)}.$$
	In particular, for any compact set $B \subset \mathbb{R}$, there are constants $c_B, C_B > 0$ such that
	\begin{equation} \label{eq:normcomp}
	c_B \|u\|_{H^{-1/2}(\gamma)} \leq \|u\|_{H^{-1/2}_{\xi}(\gamma)} \leq C_B \|u\|_{H^{-1/2}(\gamma)}, \qquad \xi \in B, \; u \in H^{-1/2}(\gamma).
	\end{equation}
\end{lemma}
\begin{proof}
By Lemma~\ref{lem:Sxicont}, there is for every $\epsilon>0$ a $\delta>0$ such that if $|\xi-\xi'|<\delta,$ then
\begin{align*}
    \|\tilde{S}(i\xi)-\tilde{S}(i\xi')\|_{\mathcal{L}(L^2(\gamma), H^1(\gamma))}= \|\tilde{S}(i\xi)-\tilde{S}(i\xi')\|_{\mathcal{L}(H^{-1}(\gamma), L^2(\gamma))}<\epsilon.
\end{align*}
By interpolation we obtain, for $|\xi - \xi'| < \delta$, that
$$
 \|\tilde{S}(i\xi)-\tilde{S}(i\xi')\|_{\mathcal{L}({H^{-1/2}(\gamma), H^{1/2}(\gamma)})}<\epsilon,
$$
and therefore that
\begin{align*}
\left| \|u\|_{H^{-1/2}_{\xi}(\gamma)}^2-\|u\|_{H^{-1/2}_{\xi'}(\gamma)}^2\right| =    |\langle(\tilde{S}(i\xi)-\tilde{S}(i\xi'))u,u \rangle_{L^2(\gamma)}| \leq \varepsilon \|u\|_{H^{-1/2}(\gamma)}^2.
\end{align*}
Since $\|u\|_{H^{-1/2}_{\xi}(\gamma)} \simeq \|u\|_{H^{-1/2}(\gamma)}$ for every fixed $\xi$, this implies \eqref{eq:normcomp} via a compactness argument. In turn, it also implies the continuity of the $H^{-1/2}_\xi(\gamma)$-norm.
 \end{proof}

\subsection{Analysis of $H(i\xi)$ on $H^{1/2}(\gamma)$}
Let $\xi \in \R$. Since $H(i\xi)$ is the double layer potential operator on $\gamma$ for $-\Delta_{{\hat{\gamma}}} + 1/4 + \xi^2$, we have the Calder\'on identity
\begin{equation}\label{plemelj}
H(i\xi) \tilde{S}(i\xi) = \tilde{S}(i\xi) H^\ast(i\xi),
\end{equation}
valid on $L^2(\gamma)$ \cite[Formula (7.41)]{MT99}. Therefore $H^\ast(i\xi)$ is formally symmetric in the scalar product of $\mathcal{E}(\gamma, -\Delta_{\hat{\gamma}} + 1/4 + \xi^2)$ (just like $K^\ast$ is symmetric in the scalar product of $\mathcal{E}(\partial \Omega)$). In fact, the symmetrization theory initiated by Krein \cite{Krein} implies that $H^\ast(i\xi)$ defines a bounded self-adjoint operator on $\mathcal{E}(\gamma, -\Delta_{\hat{\gamma}} + 1/4 + \xi^2)$, and
$$\|H^\ast(i\xi)\|_{\mathcal{L}(H^{-1/2}_\xi(\gamma))} \leq \|H^\ast(i\xi)\|_{L^2(\gamma)}.$$
Combined with Lemma~\ref{PropertiesH}, we have the following conclusion.
\begin{lemma} \label{lem:krein}
For every $\xi \in \R$, $H^\ast(i\xi)$ is self-adjoint on $\mathcal{E}(\gamma, -\Delta_{\hat{\gamma}} + 1/4 + \xi^2)$, and
$$\sup_{\xi \in \R} \|H^\ast(i\xi)\|_{\mathcal{L}(H^{-1/2}_\xi(\gamma))} < \infty.$$
\end{lemma}

We can also deduce the continuous dependence on $\xi$ from \eqref{plemelj}.
\begin{lemma}
The map $\xi \mapsto H(i\xi) \colon H^{1/2}(\gamma) \to H^{1/2}(\gamma)$ is continuous on $\R$.
\end{lemma}
\begin{proof}
From Lemmas~\ref{PropertiesH} and \ref{lem:Sxicont} we know that $H(i\xi) \colon L^2(\gamma) \to L^2(\gamma)$ and $\tilde{S}(i\xi) \colon L^2(\gamma) \to H^1(\gamma)$ depend continuously on $\xi$. Therefore the same is true of $H(i\xi) \colon H^1(\gamma) \to H^1(\gamma)$, in view of the formula
$$H(i\xi) = \tilde{S}(i\xi) H^\ast(i\xi) (\tilde{S}(i\xi))^{-1} \colon H^1(\gamma) \to H^1(\gamma).$$
The result now follows from interpolation.
\end{proof}
To state the main result of this subsection, we recall the notation of Lemma~\ref{spectH} and Theorem~\ref{contessspecK}. In particular, $\Lambda^{\alpha}_{\gamma, i\xi}$ denotes the isolated eigenvalues of $H(i\xi) \colon L^2_\alpha(\gamma) \to L^2_\alpha(\gamma)$. By Lemma~\ref{spectH}, every $\lambda \in \Lambda^{\alpha}_{\gamma, i\xi}$ is located in $(-1/2,1/2)$ and satisfies $|\lambda| > |\Sigma_{\alpha, \beta_{j^\ast}}|$. By Lemma~\ref{lem:eigdec}, the sets $\Lambda^{\alpha}_{\gamma, i\xi}$ are increasing in $0 \leq \alpha < 1$.

\begin{theorem}\label{SpecHH1/2g}
For every $\xi\in\R$,
$$\sigma\left(H(i\xi), H^{1/2}(\gamma)\right)= \left\{x\in\R \, : \, |x|\le \max_{1\le j\le J   } \frac{|1-\beta_j/\pi|}{2} \right\}\cup \Lambda^\ast_{\gamma, i\xi},$$
where the set $\Lambda^\ast_{\gamma, i\xi} \subset (-1/2, 1/2)$ consists of the isolated eigenvalues of $H(i\xi) \colon H^{1/2}(\gamma) \to H^{1/2}(\gamma)$ and satisfies $$\Lambda^\ast_{\gamma, i\xi} = \bigcup_{0 \leq \alpha < 1} \Lambda^{\alpha}_{\gamma, i\xi}.$$
Furthermore, each point $\lambda_{\xi} \in \Lambda^\ast_{\gamma, i\xi}$ depends continuously on $\xi$.
\end{theorem}
	\begin{proof}
	We refer to the decomposition \eqref{eq:Hdecomp} for $\mre z = 0$,
	$$
	H(i\xi)=H_0(i\xi)+H_1(i\xi)=\sum_{1\le j\le J}\varphi_{j}   H(i\xi)\varphi_j+H_1(i\xi).
	$$
Then $H_1(i\xi) \colon L^2(\gamma) \to L^2(\gamma)$ is a Hilbert-Schmidt operator and it is not hard to see that $H_1(i\xi) \colon H^1(\gamma) \to H^1(\gamma)$ is compact \cite[p. 120]{Elschner}. From \eqref{eq:changeofvar} and \eqref{decompH} we find a planar curvilinear polygon $\tilde{\gamma} \subset \R^2$ with the same angles as $\gamma$ and a bi-Lipschitz change of variable $\tau \colon \tilde{\gamma} \to \gamma$, inducing an operator $Q \colon H^1(\gamma) \to H^1(\tilde{\gamma})$, $Qv = v \circ \tau$, such that:
\begin{itemize}
	\item in a neighborhood of every corner, $\tilde{\gamma}$ coincides with two line segments;
	\item $H_0(i\xi) = Y + I(\xi)$ decomposes into a compact term $I(\xi) \colon L^2(\gamma) \to L^2(\gamma)$ and an operator $Y$ such that
	$$Q Y Q^{-1} = \sum_{j=1}^J \rho_j K_{\tilde{\gamma}} \rho_j,$$
	where $K_{\tilde{\gamma}}$ is the planar Neumann--Poincar\'e operator of $\tilde{\gamma}$, and $(\rho_j)_j$ are smooth cut-off functions for the corners of $\tilde{\gamma}$, with mutually disjoint supports.
\end{itemize}
Note that $H(i\xi) \colon H^1(\gamma) \to H^1(\gamma)$ is bounded as a consequence of its $L^2(\gamma)$-boundedness and \eqref{plemelj}. The operator $Q Y Q^{-1} \colon H^1(\tilde{\gamma}) \to H^1(\tilde{\gamma})$, and thus $Y \colon H^1(\gamma) \to H^1(\gamma)$, is bounded for a similar reason. We conclude that $I(\xi) \colon H^1(\gamma) \to H^1(\gamma)$ must be bounded.

Since $H_1(i\xi), I(\xi) \colon L^2(\gamma) \to L^2(\gamma)$ are compact and $H_1(i\xi), I(\xi) \colon H^1(\gamma) \to H^1(\gamma)$ are bounded, extrapolation \cite{Cwikel} (\red{see Section \ref{subsec:extrapolation}}) yields that $H_1(i\xi)$ and $I(\xi)$ are compact as operators on $H^{1/2}(\gamma)$. The planar operator $QYQ^{-1}$ has been studied in \cite[Theorem~7 and Lemma~9]{PerfektPutinar}. Since $Q$ also acts as an isomorphism $Q \colon H^{1/2}(\gamma) \to H^{1/2}(\tilde{\gamma})$ we conclude that
$$\sigma_{\ess}(H(i\xi), H^{1/2}(\gamma)) = \sigma_{\ess}(QYQ^{-1}, H^{1/2}(\tilde{\gamma})) = \left\{x\in\R \, : \, |x|\le \max_{1\le j\le J   } \frac{|1-\beta_j/\pi|}{2} \right\}.$$
The remainder of the spectrum is made up of a sequence $\Lambda^\ast_{\gamma, i\xi}$ of isolated eigenvalues, since $H^\ast(i\xi) \colon H^{-1/2}(\gamma) \to H^{-1/2}(\gamma)$ is self-adjoint in the $H^{-1/2}_\xi(\gamma)$-norm, see Lemma~\ref{lem:krein}. 

Suppose that $\lambda \in \Lambda^\ast_{\gamma, i\xi}$. By the Hardy-type inequality deduced in Remark~\ref{rmk:energynorm} (which also follows from the Rellich-Kondrachov theorem and a fractional Hardy inequality), we know that $H^{1/2}(\gamma) \subset L^2_{\alpha}(\gamma)$ for every $0 \leq \alpha < 1$. Thus $\lambda$ is an eigenvalue of $H(i\xi) \colon L^2_{\alpha}(\gamma) \to L^2_{\alpha}(\gamma)$ for every $0 \leq \alpha < 1$. Since
$$\lim_{\alpha \to 1^{-}} |\Sigma_{\alpha, \beta_j}| = \frac{|1 - \beta_j/\pi|}{2}$$
we find that $\lambda$ is an isolated eigenvalue of $H(i\xi) \colon L^2_{\alpha}(\gamma) \to L^2_{\alpha}(\gamma)$ for sufficiently large $\alpha < 1$, that is, $\lambda \in \Lambda^\alpha_{\gamma, i\xi}$. Conversely, if $\lambda \in \Lambda^\alpha_{\gamma, i\xi}$ for some $0 \leq \alpha < 1$, then, by an index argument, $\lambda$ is an eigenvalue of $H^\ast(i\xi) \colon L^2_{-\alpha}(\gamma) \to L^2_{-\alpha}(\gamma)$ and therefore of $H^\ast(i\xi) \colon H^{-1/2}(\gamma) \to H^{-1/2}(\gamma)$, since $L^2_{-\alpha}(\gamma) \subset H^{-1/2}(\gamma)$. Thus $\lambda \in \Lambda^\ast_{\gamma, i\xi}$.

The continuity of the eigenvalues follows from the same argument as in Lemma~\ref{spectH}.
\end{proof}
\subsection{Analysis on the energy space of a polyhedral cone}
As in Section~\ref{sec:energycone}, let $\Gamma$ be a Lipschitz polyhedral cone.
Since $K^\ast$ is formed with respect to the duality pairing of $L^2(\Gamma) = L^2(\R_+, r \, dr) \otimes L^2(\gamma)$, its convolution kernel is given by 
$$K^\ast(t, \omega, \omega') = \frac{1}{t^2} K\left( \frac{1}{t}, \omega', \omega \right).$$
Therefore, for $0 < \mre z < 3$,
$$\mathcal{M}K^\ast(z) = (\mathcal{M}K(2 - \bar{z}))^\ast = H^\ast(3/2 - \bar{z}).$$
In particular, when $z = i \xi + 3/2$, $\xi \in \R$, we have that
$$\mathcal{M}K^\ast(i\xi + 3/2) = H^\ast(i\xi).$$
By the identification of $\mathcal{E}(\Gamma)$ in Section~\ref{sec:energycone} and Lemma~\ref{lem:krein}, we therefore obtain, for $f,g \in C^\infty_c( \cup_j F_j)$, that
\begin{align*}
\langle K^\ast f, g\rangle_{\mathcal{E}(\Gamma)} &= \frac{1}{2\pi}\int_\R\langle H^\ast(i\xi) \mathcal{M}f(i\xi + 3/2), \mathcal{M}g(i\xi+3/2) \rangle_{H^{-1/2}_\xi(\gamma)} \, d\xi \\
&= \frac{1}{2\pi}\int_\R\langle \mathcal{M}f(i\xi + 3/2), H^\ast(i\xi)\mathcal{M}g(i\xi+3/2) \rangle_{H^{-1/2}_\xi(\gamma)} \, d\xi =\langle f,K^\ast g\rangle_{\mathcal{E}(\Gamma)}.
\end{align*}
In other words, we have the following lemma.
\begin{lemma} \label{lem:Kastuneq}
$K^\ast \colon \mathcal{E}(\Gamma) \to \mathcal{E}(\Gamma)$ is unitarily equivalent to
$$I \otimes H^\ast(i\xi) \colon L^2(\mathbb{R}, d\xi) \otimes H^{-1/2}_\xi(\gamma) \to L^2(\mathbb{R}, d\xi) \otimes H^{-1/2}_\xi(\gamma).$$
In particular, $K^\ast \colon \mathcal{E}(\Gamma) \to \mathcal{E}(\Gamma)$ is self-adjoint and bounded, by Lemma~\ref{lem:krein}.
\end{lemma}
For brevity, we write
$$|\Sigma_{\ast, \beta_j}| = \frac{|1 - \beta_j/\pi|}{2},$$
and let $j^\ast$ be an index such that $|\Sigma_{\ast, \beta_{j^\ast}}| = \max_{1 \leq j \leq J} |\Sigma_{\ast, \beta_{j}}|$.

\begin{theorem}\label{specKonSob} 
Let $\mu_\pm$ be as in Theorem~\ref{contessspecK}, and let
\begin{equation*} 
\Lambda^\ast = \{\lambda \, : \, \lambda \textnormal{ is an isolated eigenvalue of } H(z) \colon H^{1/2}(\gamma) \to H^{1/2}(\gamma), \textnormal{ for some} \mre z = 0\}.
\end{equation*}
Then
\begin{equation} \label{eq:Lambdast}
\Lambda^\ast = \left[-\mu_-, -|\Sigma_{\ast, \beta_{j^*}}|\right) \cup \left(|\Sigma_{\ast, \beta_{j^*}}|, \mu_+\right].
\end{equation}
Furthermore, $\sigma(K^*, \mathcal{E}(\Gamma)) = \sigma_{\ess}(K^*, \mathcal{E}(\Gamma))$, and

\begin{equation*}\sigma(K^*, \mathcal{E}(\Gamma)) =\left[-|\Sigma_{\ast, \beta_{j^*}}|,|\Sigma_{\ast, \beta_{j^*}}| \right]\cup \Lambda^\ast.
\end{equation*}
	\end{theorem}
\begin{proof}
By Theorem~\ref{SpecHH1/2g},
$$\Lambda^\ast = \bigcup_{\mre z = 0} \Lambda^\ast_{\gamma, z} = \bigcup_{0 \leq \alpha < 1} \Lambda^\alpha.$$
Since $\inf_{0 \leq \alpha < 1} |\Sigma_{\alpha, \beta_j}| = |\Sigma_{\ast, \beta_j}|$, \eqref{eq:Lambdast} follows from Theorem~\ref{contessspecK}.

Suppose that $\lambda \notin \left[-|\Sigma_{\ast, \beta_{j^*}}|,|\Sigma_{\ast, \beta_{j^*}}| \right]\cup \Lambda^\ast$. Then, by Theorem~\ref{SpecHH1/2g} and since $H^\ast(i\xi)$ is self-adjoint on $H^{-1/2}_{\xi}(\gamma)$,
$$\sup_{\xi \in \R} \|(H^\ast(i\xi) - \lambda)^{-1}\|_{\mathcal{L}(H^{-1/2}_{\xi}(\gamma))} = \sup_{\xi \in \R}\frac{1}{\dist (\lambda, \sigma(H(i\xi), H^{1/2}(\gamma)))} < \infty.$$
Therefore, by Lemma~\ref{lem:tensorbdd}, $I \otimes (H^\ast(i\xi) - \lambda)^{-1}$ defines a bounded inverse of 
$$I \otimes (H^\ast(i\xi) - \lambda) \colon L^2(\mathbb{R}, d\xi) \otimes H^{-1/2}_\xi(\gamma) \to L^2(\mathbb{R}, d\xi) \otimes H^{-1/2}_\xi(\gamma).$$
Hence $\lambda \notin \sigma(K^\ast, \mathcal{E}(\Gamma))$, by Lemma~\ref{lem:Kastuneq}.
	
Conversely, suppose that $\lambda \in \left[-|\Sigma_{\ast, \beta_{j^*}}|,|\Sigma_{\ast, \beta_{j^*}}| \right]\cup \Lambda^\ast$. Then, again by self-adjointness, there is a $\xi \in \R$ such that $\lambda$ is either an eigenvalue of $H^\ast(i\xi) \colon H^{-1/2}_{\xi}(\gamma) \to H^{-1/2}_{\xi}(\gamma)$, or there is a singular Weyl sequence for $H^\ast(i\xi) - \lambda$. Then we can follow the arguments of Lemma~\ref{lem:Weyl} and Theorem~\ref{specK} with minor modifications in order to construct a singular Weyl sequence for $K^\ast - \lambda \colon \mathcal{E}(\Gamma) \to \mathcal{E}(\Gamma)$, showing that $\lambda \in \sigma_{\ess}(K^\ast, \mathcal{E}(\Gamma))$.
\end{proof}

When the polyhedral cone is convex, we can obtain additional information about $\mu_+$.
\begin{theorem} \label{thm:coneconvex}
Suppose that the polyhedral cone $\hat{\Gamma}$ is convex. Then $\mu_- \leq \mu_+$ and 
$$\mu_+ = \max \sigma(H^\ast(0), H^{-1/2}(\gamma)).$$
\end{theorem}
\begin{proof}
Suppose that $\lambda \in \Lambda^\ast_{\gamma, i\xi}$ for some $\xi \in \R$. Then, by Theorem~\ref{SpecHH1/2g}, there is an $0 \leq \alpha < 1$ such that $\lambda \in \Lambda^\alpha_{\gamma, i\xi}$. Therefore there is a function $g \in L^2_{-\alpha}(\gamma) \subset H^{-1/2}(\gamma)$ in the kernel of $H^\ast(i\xi) - \lambda$.

Since $\hat{\Gamma}$ is convex, the convolution kernel of $K^\ast$ satisfies $K^\ast(t, \omega, \omega') \geq 0$ for all $t \in \R_+$ and $\gamma, \gamma' \in \omega$, and therefore
$$
|H^\ast(i\xi)(\omega,\omega')| = \left| \int_0^\infty t^{i\xi + 3/2} K^\ast(t, \omega, \omega') \, \frac{dt}{t} \right|  \le H^\ast(0) (\omega,\omega'), \quad \omega,\omega'\in\gamma.
$$
In particular,
$|\lambda| |g| = |H^\ast (i\xi) g| \le H^\ast(0) |g|$.
Noting that $|g| \in L^2_{-\alpha}(\gamma)\subset H^{-1/2}(\gamma)$ and that the kernel of $\tilde{S}(0)$ is positive, we find that
$$
\langle H^\ast(0)|g|, |g| \rangle_{H^{-1/2}_0(\gamma)} \geq |\lambda| \| \, |g| \, \|_{H^{-1/2}_0(\gamma)}^2 .
$$
Since $H^\ast(0) \colon H^{-1/2}_0(\gamma) \to H^{-1/2}_0(\gamma)$ is self-adjoint, it follows from the min-max principle that $|\lambda| \leq \max \Lambda^\ast_{\gamma, 0}$. This proves the theorem.
\end{proof}

\subsection{Localization in the energy space}
Let $\partial \Omega$ be the boundary of a Lipschitz polyhedron, and let $K$ be the associated Neumann--Poincar\'e operator. Before proving a localization result for $K \colon H^{1/2}(\partial \Omega) \to H^{1/2}(\partial \Omega)$, we need a number of lemmas. Recall that the energy space $\mathcal{E}(\partial \Omega)$ is isomorphic to $H^{-1/2}(\partial \Omega)$.

\begin{lemma}\label{commcompact2}
Let $\varphi$ be a Lipschitz function on $\partial \Omega$. Then the commutator $[K^\ast,\varphi]=K^\ast\varphi-\varphi K^\ast$ is compact on $\mathcal{E}(\partial\Omega)$.
\end{lemma}
\begin{proof}
$[K^\ast,\varphi] \colon L^2(\partial\Omega) \to L^2(\partial\Omega)$ is compact, since the kernel of $[K^\ast,\varphi]$ is weakly singular. Furthermore, $[K^\ast,\varphi] \colon H^{-1}(\partial \Omega) \to H^{-1}(\partial \Omega)$ is bounded, since  $K:H^1(\partial\Omega)\to H^1(\partial\Omega)$ is bounded \cite{Verchota}. From extrapolation \cite{Cwikel} \red{(see Section \ref{subsec:extrapolation})} we therefore conclude that  $[K^\ast,\varphi]$ is compact on $H^{-1/2}(\partial\Omega) \simeq \mathcal{E}(\partial \Omega)$.
\end{proof}
To utilize our understanding of the adjoint Neumann--Poincar\'e operator $K^\ast_\Gamma \colon \mathcal{E}(\Gamma) \to \mathcal{E}(\Gamma)$ for Lipschitz polyhedral cones $\Gamma$, we also need to recall some aspects of \cite[Section~4]{Per19}.We may assume that $\Gamma$ is of the form
$$\Gamma = \{(x', \phi(x')) \, : \, x' \in \R^2\},$$
where $\phi \colon \R^2 \to \R$ is Lipschitz continuous. For a function $f$ on $\Gamma$, we define $\Pi f$ as the function on $\R^2$ for which
$$\Pi f (x') = f(x', \phi(x')), \qquad x' \in \R^2.$$
For $0 \leq s \leq 1$, the homogeneous Sobolev space $\dot{H}^s(\R^2)$ is the completion of $C_c^\infty(\R^2)$ in the norm
$$\|f\|_{\dot{H}^s(\R^2)}^2 = \int_{\R^2} |\mathcal{F} f(u) |^2 |u|^{2s} \, du,$$
where $\mathcal{F} \colon L^2(\R^2) \to L^2(\R^2)$ denotes the usual Fourier transform. For $0 \leq s < 1$ the Sobolev space $\dot{H}^s(\R^2)$ is a space of functions; it is continuously contained in $L^{2/(1-s)}(\R^2)$. For $s = 1$, $\dot{H}^1(\R^2)$ is the space of functions modulo constants such that $\nabla f \in L^2(\R^2)$. For $0 \leq s \leq 1$ we define $\dot{H}^s(\Gamma)$ by
$$\dot{H}^s(\Gamma) = \Pi^{-1} \dot{H}^s(\R^2),$$
and for $-1 \leq s < 0$ we let $\dot{H}^s(\Gamma)$ be the dual of $\dot{H}^{-s}(\Gamma)$ with respect to the $L^2(\Gamma)$-pairing. In this notation, the content of \cite[Theorem~14]{Per19} is that $\mathcal{E}(\Gamma)$ coincides with $\dot{H}^{-1/2}(\Gamma)$,
$$\mathcal{E}(\Gamma) \simeq \dot{H}^{-1/2}(\Gamma).$$

\begin{lemma} \label{lem:offdiagcpct}
Suppose that $\varphi$ and $\eta$ are two compactly supported Lipschitz functions on $\Gamma$ such that $1-\eta$ and $\varphi$ have disjoint support. 
Then
$$(1-\eta) K^\ast_\Gamma \varphi \colon \mathcal{E}(\Gamma) \to \mathcal{E}(\Gamma)$$
is compact.
\end{lemma}
\begin{proof}
By duality, we may equivalently prove that
\begin{equation} \label{eq:Tcpct}
T := \varphi K_\Gamma (1-\eta) \colon \dot{H}^{1/2}(\Gamma) \to \dot{H}^{1/2}(\Gamma)
\end{equation}
is compact. Applying H\"older's inequality, it is straightforward to check, for $f \in L^4(\Gamma)$, that
$$\|Tf\|_{L^\infty(\Gamma)} \lesssim \|f\|_{L^4(\Gamma)}$$
and
$$|Tf(x) - Tf(y)| \lesssim \|f\|_{L^4(\Gamma)}|x-y|, \qquad x, y \in \Gamma.$$
Since $\dot{H}^{1/2}(\Gamma) \subset L^4(\Gamma)$, we find that $T$ is bounded as an operator from $\dot{H}^{1/2}(\Gamma)$ into the zero trace Sobolev space $H^1_0(U_\varphi)$, where $U_\varphi$ is any bounded open set such that $\supp \varphi \subset U_\varphi$. It follows that the operator $T$ of \eqref{eq:Tcpct} is compact.
\end{proof}

We now provide our final theorem.
    \begin{theorem}
let $K$ be the Neumann--Poincar\'e operator of a Lipschitz polyhedron $\partial \Omega$. For each vertex of $\partial \Omega$, let $K_i = K_{\Gamma_i}$ denote the Neumann--Poincar\'e operator of the corresponding tangent polyhedral cone $\Gamma_i$, $i = 1, \ldots, I$.  

Then, for $\lambda \in \mathbb{C}$, $K^\ast - \lambda$ is Fredholm on $\mathcal{E}(\partial \Omega) \simeq H^{-1/2}(\partial \Omega)$ if and only if $K_i^\ast - \lambda$ is invertible on $\mathcal{E}(\Gamma_i)$ for every $i=1,\dots,I$. That is,
$$\sigma_{\ess}(K^\ast, \mathcal{E}(\partial\Omega)) = \bigcup_{1 \leq i \leq I} \sigma(K_i^\ast, \mathcal{E}(\Gamma_i)).$$
The spectra $\sigma(K_i^\ast, \mathcal{E}(\Gamma_i))$ have been described in Theorem~\ref{specKonSob}.
	\end{theorem}
\begin{remark} \rm
Since $K^\ast \colon \mathcal{E}(\partial \Omega) \to \mathcal{E}(\partial \Omega)$ is self-adjoint and $\mathcal{E}(\partial \Omega) \simeq H^{-1/2}(\partial \Omega)$, we obtain as a corollary that
$$\sigma(K, H^{1/2}(\partial \Omega)) = \bigcup_{1 \leq i \leq I} \sigma(K_i^\ast, \mathcal{E}(\Gamma_i)) \cup \{\lambda_k\}_k ,$$
where $\{\lambda_k\}_k$ is a sequence of real isolated eigenvalues.
\end{remark}
	\begin{proof}
	We follow the proof of Theorem~\ref{locarg1}, retaining its notation.
	
Suppose that $\lambda I-K_i^\ast$ is invertible on $\mathcal{E}(\Gamma_i)$ for every $i=1,\dots,I.$ Then, for $f\in \mathcal{E}(\partial\Omega)$,
	\begin{align*}
	\|f\|_{\mathcal{E}(\partial\Omega)}&\simeq\sum_{1\le i\le I}\|\varphi_i f\|_{\mathcal{E}(\Gamma_i)}\lesssim \sum_{1\le i\le I}\|(\lambda I- K_i^\ast)\varphi_i f\|_{\mathcal{E}(\Gamma_i)}\\
	&\lesssim \sum_{1\le i\le I}\|{\eta_i}(\lambda I- K^\ast_i)\varphi_i f\|_{\mathcal{E}(\Gamma_i)}+  \sum_{1\le i\le I}\|{(1-\eta_i)}(\lambda I- K_i^\ast)\varphi_i f\|_{\mathcal{E}(\Gamma_i)}\\
	&\lesssim  \sum_{1\le i\le I}\|(\lambda I- K^\ast)\varphi_i f\|_{\mathcal{E}(\partial\Omega)}+ \sum_{1\le i\le I}\|{(1-\eta_i)} K_i^\ast \varphi_i f\|_{\mathcal{E}(\Gamma_i)},
	\end{align*}
	where we have used that $1 - \eta_i$ and $\varphi_i$ have disjoint support, and that
	$$\|{\eta_i}(\lambda I- K^\ast_i)\varphi_i f\|_{\mathcal{E}(\Gamma_i)} \simeq \|{\eta_i}(\lambda I- K^\ast)\varphi_i f\|_{H^{-1/2}(\partial \Omega)} \lesssim \|(\lambda I- K^\ast)\varphi_i f\|_{\mathcal{E}(\partial\Omega)}.$$
	By Lemmas~\ref{commcompact2} and \ref{lem:offdiagcpct}, there is thus a Hilbert space $\mathcal{H}$ and a compact operator $C \colon \mathcal{E}(\partial\Omega) \to \mathcal{H}$ such that
	\begin{align*}
	\|f\|_{\mathcal{E}(\partial\Omega)}&\lesssim \sum_{1\le i\le I}\|\varphi_i(\lambda I- K^\ast) f\|_{\mathcal{E}(\partial\Omega)}+ \|C f\|_{\mathcal{H}}\\
	&\lesssim \|(\lambda I- K^\ast) f\|_{\mathcal{E}(\partial\Omega)} + \|C f\|_{\mathcal{H}}.
	\end{align*}
	This shows that $\lambda I-K^\ast$ is Fredholm on $\mathcal{E}(\partial \Omega)$, since $K^\ast \colon \mathcal{E}(\partial \Omega) \to \mathcal{E}(\partial \Omega)$ is self-adjoint.
	   
	 Conversely, suppose that $\lambda I-K_{i_0}^\ast \colon \mathcal{E}(\Gamma_{i_0}) \to \mathcal{E}(\Gamma_{i_0})$ fails to be invertible for some $i_0\in \{1,\dots,I\}$. Then, by the proofs of Theorems~\ref{specK} and \ref{specKonSob}, there is a singular Weyl sequence $(w_n)$ for $\lambda I-K_{i_0}^\ast \colon \mathcal{E}(\Gamma_{i_0}) \to \mathcal{E}(\Gamma_{i_0})$, supported in a sufficiently small neighborhood of the vertex of $\Gamma_{i_0}$. 	
	We interpret $(w_n)$ as a sequence in $\mathcal{E}(\partial \Omega)$, tending to $0$ weakly and satisfying that
	$\|w_n\|_{\mathcal{E}(\partial \Omega)} = 1$ for all $n$, cf. Section~\ref{subsec:energyspace}.
	
	Choosing the support of $w_n$ appropriately, we have the following equation in $\mathcal{E}(\partial \Omega)$,
	$$(\lambda I - K^\ast)w_n = \eta_{i_0} (\lambda I - K^\ast) \varphi_{i_0} w_n - (1-\eta_{i_0}) K^\ast \varphi_{i_0} w_n.$$
	It is easy to verify that the operator $(1-\eta_{i_0}) K^\ast \varphi_{i_0} \colon \mathcal{E}(\partial \Omega) \to \mathcal{E}(\partial \Omega)$ is compact, either by extrapolation or by arguing as in the proof of Lemma~\ref{lem:offdiagcpct}. Therefore $(1-\eta_{i_0}) K^\ast \varphi_{i_0} w_n \to 0$ in $\mathcal{E}(\partial \Omega)$ as $n \to \infty$, since $w_n \to 0$ weakly. Next we understand $\eta_{i_0} (\lambda I - K^\ast) \varphi_{i_0} w_n$ as an element of $\mathcal{E}(\Gamma_{i_0})$ satisfying
	$$\eta_{i_0} (\lambda I - K^\ast) \varphi_{i_0} w_n = (\lambda I- K_{i_0}^\ast) w_n + (1 - \eta_{i_0}) K_{i_0}^\ast\varphi_{i_0} w_n.$$
	Here $(\lambda I- K_{i_0}^\ast) w_n \to 0$ by the choice of $(w_n)$ as a Weyl sequence, while $(1 - \eta_{i_0})K_{i_0}^\ast \varphi_{i_0} w_n \to 0$ in $\mathcal{E}(\Gamma_{i_0})$ by Lemma~\ref{lem:offdiagcpct}. 
	
	We have shown that $(\lambda I - K^\ast)w_n \to 0$ in $\mathcal{E}(\partial \Omega)$, demonstrating that $\lambda \in \sigma_{\ess}(K^\ast, \mathcal{E}(\partial\Omega))$. 
\end{proof}	
 
{\bf Declarations}

\thanks{
	 The authors were supported by grant EP/S029486/1 of the Engineering and Physical Sciences Research Council (EPSRC), and the second author 
	was partially supported by the ERCIM `Alain Bensoussan' Fellowship Programme.}


	
	\bibliographystyle{amsplain-nodash} 

\end{document}